%% file: main.tex
\documentclass[10pt]{amsart}
\input{preamble}

\DefCurrentAudience{shareable}

\title{Arithmetic intersections on non-split Cartan modular curves}
\author{Jonathan Love, \'{E}lie Studnia, and Jan Vonk}

\email{\texttt{j.love@math.leidenuniv.nl}\textnormal{, Mathematical Institute, University of Leiden,  2333 CC Leiden}}
\email{\texttt{e.n.g.studnia@math.leidenuniv.nl}\textnormal{, Mathematical Institute, University of Leiden,  2333 CC Leiden}}
\email{\texttt{j.b.vonk@math.leidenuniv.nl}\textnormal{, Mathematical Institute, University of Leiden,  2333 CC Leiden}}

\begin{document}

\begin{abstract}
Let $p$ be a prime number, and let $\Delta_1,\Delta_2 < 0$ be two coprime fundamental discriminants. When $p$ \textit{splits} in $\Q(\sqrt{\Delta_1})$ and $\Q(\sqrt{\Delta_2})$ the height pairings of the corresponding CM divisors on $X_{\spl}^+(p)$ were determined by Gross--Kohnen--Zagier \cite{GKZ87}. When $p$ is \textit{inert}, we determine the arithmetic intersection numbers of the corresponding divisors on $X_{\ns}^+(p)$ at all finite primes. The key point of our analysis is at the prime of bad reduction $p$: to determine the intersection numbers at $p$, we provide a moduli interpretation for the smooth locus in the regular model of $X_{\ns}^+(p)$ over $\mathrm{Spec}(\Z)$ constructed by Edixhoven--Parent \cite{EP24}.
\end{abstract}

\maketitle
\tableofcontents

\section{Introduction}

The work of Heegner \cite{Hee52} reduces the classification of negative discriminants of class number one to the determination of integer solutions to the Diophantine equation
\begin{equation}
\label{eqn:Xns24}
y^2 = x^3 + 8.
\end{equation}
As explained by Serre \cite[Appendix]{Ser97}, the strategy of Heegner may be interpreted, and framed more generally, as the determination of the set of rational points, integral with respect to the cusps, on a non-split Cartan modular curve $X_{\ns}^+(N)$. The case employed by Heegner is $X_{\ns}^+(24)$ which is an elliptic curve described by the equation \cref{eqn:Xns24}. Siegel \cite{Sie68} similarly considered the elliptic curve $X_{\ns}^+(15)$, and used it to give an alternative proof\footnote{The first proof with prime level was given by Kenku \cite{Ken85} for $N=7$. The curve $X_{\ns}^+(11)$ is an elliptic curve of rank one, see Ligozat \cite{Lig77}, and has also been used to resolve the class number one problem by Schoof--Tzanakis \cite{ST12}. The rational points on $X_{\ns}^+(13)$ and $X_{\ns}^+(17)$, curves of genus $3$ and $6$ respectively, have also been determined \cite{BDMTV19,BDMTV23}. } of the class number one problem.

\par In spite of this momentous early application of CM points on non-split Cartan curves, our knowledge on rational points on $X_{\ns}^+(p)$ remains limited. It is the sole remaining obstacle for establishing Serre uniformity \cite{Ser72}, which asks whether there exists a constant $C$ such that the Galois representation
\[
\rho_{E,p} : \mathrm{Gal}(\overline{\Q}/\!\Q) \ \lra \ \mathrm{Aut}_{\F_p}(E[p]) \simeq \mathrm{GL}_2(\F_p)
\]
is surjective for all non-CM elliptic curves $E/\Q$, and all primes $p > C$. It is widely believed that $C=37$ suffices. The work of Serre \cite{Ser72}, Mazur \cite{Maz77}, Bilu--Parent--Rebolledo \cite{BP11,BPR13} reduces this to the question of whether $X_{\ns}^+(p)(\Q)$ consists entirely of CM points, for $p$ large enough. 

This paper makes a study of the arithmetic geometry of non-split Cartan modular curves, and their CM points. Our first result concerns a moduli interpretation over $\mathrm{Spec}(\Z)$ of the minimal regular model.

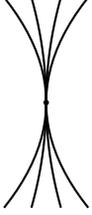
\begin{minipage}{.1\textwidth}
\begin{tikzpicture}[scale=.7]
  \def\a{0.15}

  \draw[thick, domain=-2:2, variable=\t]
    plot ({.2*\t*\t},{\t});

  \draw[thick, domain=-2:2, variable=\t]
    plot ({-.07*\t*\t},{\t});

  \draw[thick, domain=-2:2, variable=\t]
    plot ({.07*\t*\t},{\t});

  \draw[thick, domain=-2:2, variable=\t]
    plot ({-.2*\t*\t},{\t});
    
  \fill (0,0) circle (1.5pt);
\end{tikzpicture}

\end{minipage}\hfill
\begin{minipage}{.84\textwidth}
Edixhoven--Parent \cite{EP24} determine a minimal regular model for $X_{\ns}^+(p)$ over $W(\overline{\F}_p)$. It is constructed from the models of Katz--Mazur \cite{KM85} for $X(p)$ via a sequence of blow-ups, quotients, and blow-downs. Its special fibre at $p$ consists of a collection of projective lines, indexed by the supersingular $j$-invariants, which all meet at one point. They remark:\vspace{.1cm}

\begin{quote}
\textit{``Those $\PP^1$s [...] in the non-split cases have no modular interpretation, as they just result from blow-ups in the course of the regularization process.''}
\end{quote}
\end{minipage}

\par \noindent \Cref{sec:moduli} offers a moduli interpretation for the smooth locus of this model. It relies on the observation that an elliptic curve with non-split Cartan structure defined over $\Q_p^{\rm ur}$ has supersingular reduction if $p>7$. More generally, for any prime $p$ and $n\geq 1$, we study the coarse moduli space over $\mathrm{Spec}(\Z)$ of the functor
\begin{equation}
\label{eqn:moduli-functor}
\begin{array}{lcll}
\cM_{\ns}^+(p^n) : & 
\boldsymbol{\underline{\mathrm{Sch}}}_{\Z} & \lra & \boldsymbol{\underline{\mathrm{Set}}}\\
 & S &\longmapsto & \{ E/S \ \ \mbox{with} \ \ \Z_{p^2}\!/(p^n)\hookrightarrow \mathrm{End}_S(E[p^n]) \}/\sim  
\end{array}
\end{equation}
of isomorphism classes of elliptic curves $E/S$ with an embedding of $\Z_{p^2}\!/(p^n)$ into the endomorphism ring of the group scheme $E[p^n]$ over $S$, with conjugate embeddings identified. Here, $\Z_{p^2}$ is the Witt ring of $\F_{p^2}$.

\par The special fibre is a disjoint union of $\A^1$'s, which may be explained as follows. For a supersingular elliptic curve $E / S/ \F_p$, $E[p]$ represents a non-trivial class in $\mathrm{Ext}^1_S(\boldsymbol{\alpha}_p,\boldsymbol{\alpha}_p)$. Given an element $\lambda$ of $\mathrm{End}_S(\boldsymbol{\alpha}_p/S) \simeq \A^1(S)$ we obtain an endomorphism of $E[p^n]$ by composition
\begin{equation}
\label{eqn:alpha_p}
E[p^n] \ \stackrel{p^{n-1}}{\lra} \  
E[p] \lra \boldsymbol{\alpha}_p \stackrel{\lambda}{\lra} \boldsymbol{\alpha}_p \lra E[p] \ \lra \ E[p^n].
\end{equation}
This induces an action of $\A^1(S)$ on the set $\{ \Z_{p^2}\!/(p^n)\hookrightarrow \mathrm{End}_S(E[p^n])\}$ by adding this endomorphism to the image of a fixed algebra generator of $\Z_{p^2}\!/(p^n)$. The first main result is: 

\begin{thmx}\label{thm:A}
Let $p$ be a prime number. The coarse moduli scheme $\cX_{\ns}^+(p^n)$ associated with \cref{eqn:moduli-functor} is smooth over $\mathrm{Spec}(\Z)$. Its geometric special fibre at $p$ is the disjoint union of finitely many\footnote{\label{footnote:count} When $p \geq 5$, the number of components is the sum of $1 + \frac{2(p^{2n-2}-1)}{|\mathrm{Aut}(E)|}$ over all supersingular curves $E/\overline{\F}_p$, see \cref{rmk:component count}.} components isomorphic to $\A^1$.\vspace{-.3cm} 
\end{thmx}
\newcounter{mycounter}
\setcounter{mycounter}{\thefootnote}
\noindent When $n=1$ the set of components in the special fibre are canonically in bijection with the isomorphism classes of supersingular elliptic curves $E$ over $\overline{\F}_p$. In \cref{subsec:EP} we show that $\cX_{\ns}^+(p)$ is isomorphic to the smooth non-cuspidal locus of the minimal regular model constructed by Edixhoven--Parent for $p>7$.

\par The second result of this paper is an explicit determination of the arithmetic intersection of a pair of Heegner points on $\cX_{\ns}^+(p^n)$, similar to the result of Gross--Kohnen--Zagier \cite{GKZ87}. 
Let $\Delta_1,\Delta_2 < 0$ be a pair of coprime fundamental discriminants such that $p$ is \textit{inert} in both $K_1 := \Q(\sqrt{\Delta_1})$ and $K_2 := \Q(\sqrt{\Delta_2})$. There are two associated Heegner divisors, of degrees equal to the class numbers of $K_1$ and $K_2$:
\[
P_1,P_2 \ \in \  \mathrm{Div}_{\Q}(X_{\ns}^+(p^n)).
\]
Their Zariski closures in the integral model $\cX_{\ns}^+(p^n)$ over $\Z$ constructed in \cref{thm:A} intersect properly, and we explicitly determine their arithmetic intersection number 
\begin{equation}
\label{eqn:arith-int-def}
\langle P_1, P_2 \rangle \ := \sum_{x \in \mathrm{Max}\cX_{\ns}^+(p^n)}  \mathrm{length}\left( \frac{\cO_{\cX,x}}{(P_1,P_2)} \right) \cdot \log |k_x| 
\end{equation}
where $x$ ranges over closed points on $\cX_{\ns}^+(p^n)$, and we denote $k_x$ for its residue field. Using \cref{thm:A}, we reduce the problem to counting, for each prime $q$, embeddings of quadratic orders 
\[
\cO_1, \cO_2 \ \hookrightarrow \ R \subset B_{q\infty}, \qquad \mathrm{disc}(\cO_i) = \Delta_i. 
\]
Instead of the Eichler orders arising in \cite{GKZ87}, we are faced with \emph{non-split Cartan orders} $R$ (\cref{def:mixed cartan order}). These orders are not Eichler, but they are Bass (or basic) orders in the sense of \cite{CSV21}. To state the explicit formula for the intersection number $\langle P_1, P_2 \rangle$, we define for any positive integers $a,b \in \Z$ the finite set 
\begin{equation}
\label{eqn:S-def}
\mathcal{S}(a,b) :=  \Z_{>0} \ \cap \ \left\{\frac{a - x^2}{4b} : x \in \Z \right\}.
\end{equation}
We define the map $\epsilon$ as in Gross--Zagier \cite[p. 192]{GZ85}: Let $F = \Q(\sqrt{\Delta_1\Delta_2})$ and $\chi$ the genus character of $F$ associated to $K_1K_2/F$. For any prime $q$ define $\epsilon(q) := \chi(\mathfrak{q})$ where $\mathfrak{q}$ is any prime in $F$ above $q$. Extend the definition of $\epsilon$ multiplicatively. The following result is proved in \cref{sec:GZ}: 
\begin{thmx} \label{thm:B}
The intersection number of $P_1$ and $P_2$ on the arithmetic surface $\cX_{\ns}^+(p^n)$ is given by
\vspace{4pt}\begin{equation}
\label{eqn:GZformula}
\langle P_1, P_2 \rangle \ \ = \ \ \ \frac{w_1w_2}{4} \ \ \cdot \!\!\!\!\sum_{\substack{m \in \mathcal{S}(\Delta_1\Delta_2,\;p^{2n}) \\ d \, \mid \, m,\;d\,>\,0}} \epsilon\left(\frac{m}{d}\right) \log(d), \qquad \mbox{where} \ \ w_i := |\cO_i^{\times}|. 
\end{equation}
\end{thmx}\vspace{-4pt}
When $p$ \textit{splits} in the quadratic fields $K_1$ and $K_2$ the arguments of Gross--Kohnen--Zagier \cite{GKZ87} show that the intersection numbers of the Heegner divisors on the \textit{split} Cartan normaliser curve $X_{\spl}^+(p^n)$ are given by the same formula. A notable difference is that in the split case, the CM elliptic curves have \textit{ordinary} reduction at $p$, and thus cannot intersect in the bad fibre. In the inert case however, the elliptic curves are \textit{supersingular} at $p$ and therefore may (and frequently do) intersect in the bad fibre.

\begin{example}
\label{subsec:examples}
Consider $X_{\ns}^+(11)$, which is an elliptic curve with minimal Weierstra{\ss} model
\begin{equation}
\label{eqn:Xns11}
    X_{\ns}^+(11) : y^2 + y = x^3 - x^2 - 7x + 10.
\end{equation}
The minimal regular model at $11$ has Kodaira type III, with one component for each of the two supersingular $j$-invariants.
The Heegner discriminants $(\Delta_1, \Delta_2) = (-115,-267)$ give rise to points
\[
P_1 = \left(\sqrt{5},\,\frac{-5+\sqrt{5}}2\right), \qquad
P_2 = \left(\frac{33+\sqrt{89}}{25},\,\frac{189-17\sqrt{89}}{250}\right),
\] 
which have a simple intersection at $11$ on the $j=1$ component. There is another simple intersection at the prime $5$. This is precisely as predicted by \cref{thm:B}, since 
$\cS(\Delta_1\Delta_2,11^2) = \{20,44\}$
and the sum on the right hand side of \cref{eqn:GZformula} evaluates to
\[\log\left(\frac{2\cdot 5\cdot 20}{4\cdot 10}\right)+\log\left(\frac{2\cdot 11\cdot 44}{4\cdot 22}\right)=\log({55}).\]
\end{example}

\par In the body of this paper, a more general version of \cref{thm:B} is proved for modular curves with split Cartan level at some primes, and non-split Cartan level at other primes; this is \cref{thm:C}. Let $N_{\spl}$ and $N_{\ns}$ be coprime positive integers, and consider $\Gamma \leq \SL_2(\Z)$ of the form $\Gamma := \Gamma_{\spl}^+(N_{\spl}) \cap \Gamma_{\ns}^+(N_{\ns})$, see \cref{subsec:mixed-moduli}. 
A modular curve associated to such a level structure will be called \textit{mixed Cartan} and denoted 
\begin{equation}
\label{eqn:mixed-cartan}
X_{\rm Car}^+(N_{\spl},N_{\ns}).
\end{equation}
We will use the phrase \textit{Heegner discriminant for $\Gamma$} for any discriminant $\Delta < 0$ that satisfies 
\begin{equation}
\label{eqn:Heegner-disc}
\displaystyle 
\left(\frac{\Delta}{q}\right) = 1 \ \ \mbox{for all} \  \ q \,|\, N_{\spl}, \hspace{1cm} \mbox{and} \hspace{1cm}
\left(\frac{\Delta}{q}\right) = -1 \ \ \mbox{for all} \ \ q \, |\, N_{\ns}. \\ 
\end{equation}
The reader should not confuse this with common usage of "Heegner hypothesis" to mean that all primes dividing the level are split. Note that Heegner \cite{Hee52} was working with primes that are inert, and our terminology is also used in other works on non-split Cartan curves, for instance Kohen--Pacetti \cite{KP16}. 

\par \Cref{subsec:RW} complements the work of Rebolledo--Wuthrich \cite{RW25}: we determine all arithmetic intersections of any pair of CM points in $X_{\ns}^+(p)(\Q)$, for any prime $p$. \Cref{fig:Xnsp_intersections} contains a complete list of non-trivial arithmetic intersections, proving the expectation of \cite{RW25} about the completeness of their data.

Finally, we mention that our main results might be of interest for the study of rational points on $X_{\ns}^+(p)$ using the geometric quadratic Chabauty method of Edixhoven--Lido \cite{EL23}. The application of the determination of intersection numbers of Heegner points in this context was brought to our attention by Parent, and we refer to the forthcoming paper \cite{HKLLMP} for more details on this topic.

\subsection*{Acknowledgements} We would like to thank Bas Edixhoven and Henri Darmon for several inspiring conversations on the themes of this paper over the years. All three authors were supported by ERC Starting Grant 101076941 (`\textsc{GAGARIN}'). Views and opinions expressed are however those of the author(s) only and do not necessarily reflect those of the European Union or the European Research Council. Neither the European Union nor the granting authority can be held responsible for them.


\section{The moduli space $\cX_{\ns}^+(p^n)$ over $\mathrm{Spec}(\Z)$}
\label{sec:moduli}

The main purpose of this section is to prove \cref{thm:A}. We do this in greater generality and allow for an extra level structure, assumed to be finite \'etale at $p$. Our discussion freely uses the formalism of moduli problems developed in Chapter 4 of \cite{KM85}. Henceforth, we fix a prime $p$ and an integer $n \geq 1$. 

We introduce the moduli problems attached to non-split Cartan structures, and study the geometry of their coarse moduli spaces. In characteristic $p$, only supersingular elliptic curves appear in this moduli space. We begin by describing the $p^n$-torsion of supersingular elliptic curves, and their endomorphism rings. 

\subsection{The $p^n$-torsion of a supersingular elliptic curve}
We first recall some properties of the $p^n$-torsion schemes of supersingular elliptic curves in characteristic $p$. In this subsection, fix an elliptic curve $E$ over a perfect field $k$ of characteristic $p$. Recall that we have an exact sequence 
\begin{equation}
\label{eqn:ext-Ep}
0 \rightarrow \boldsymbol{\alpha}_p \ \overset{\iota}{\lra} \ E[p] \ \overset{\pi}{\lra} \ \boldsymbol{\alpha}_p \rightarrow 0,
\end{equation}
where the closed subgroup $\boldsymbol{\alpha}_p$ is the kernel of the Frobenius isogeny $E \lra E^{(p)}$. 

\begin{lemma}
\label{lemma:supersing-lie} 
For any $k$-algebra $R$, we have: 
\begin{itemize}[noitemsep,label=\tiny$\bullet$]
\item any endomorphism of $E[p]$ over $R$ restricts through $\iota$ to an endomorphism of $\boldsymbol{\alpha}_p$ over $R$,
\item the following natural maps on Lie algebras are isomorphisms of free $R$-modules of rank one:
\[
\mathrm{Lie}(\boldsymbol{\alpha}_p/R) \ \overset{\iota}{\lra} \ \mathrm{Lie}(E[p]/R) \ \lra \ \mathrm{Lie}(E[p^n]/R)\ \lra \ \mathrm{Lie}(E/R),
\]
\item there is an isomorphism of $\F_p$-algebras given by the natural map 
\[
\mathrm{End}_R(\boldsymbol{\alpha}_p/R) \ \stackrel{\sim}{\lra} \ \mathrm{End}(\mathrm{Lie}(\boldsymbol{\alpha}_p/R)) \cong R.\]
\end{itemize}

\begin{proof}
The kernel $G$ of the Frobenius isogeny is a $k$-group scheme of degree $p$ by e.g. \cite[(12.2.1)]{KM85}. Since $E[p]$ is Cartier-dual to itself and $E[p](\overline{k})=\{0\}$, $G$ is not \'etale nor Cartier-dual to a finite \'etale $k$-group scheme. Therefore, by the Oort--Tate classification, $G \simeq \boldsymbol{\alpha}_p$.   

Since $\iota$ is also the kernel of Frobenius on $E[p]$, endomorphisms of $E[p]_R$ can be restricted along $\iota$ by \cite[Exp. $\mathrm{VII}_{\mathrm{A}}$, 4.1.3]{SGA3I}. By \cite[Prop. 1.1 (a)]{LLR04}, the given maps of Lie algebras are injective and it is enough to show that Frobenius kills $\mathrm{Lie}(E/R)$: this is contained in \cite[Exp. $\mathrm{VII}_{\mathrm{A}}$, 4.1.2]{SGA3I}.

The Frobenius $\mathbb{A}^1_k \rightarrow \mathbb{A}^1_k$ is a morphism of ring schemes, so its kernel $\boldsymbol{\alpha}_p$ is a module over the ring scheme $\mathbb{A}^1_k$. It is not difficult to check that for any $k$-algebra $R$, the composition 
\[
R = \mathbb{A}^1_k(R) \lra \mathrm{End}(\boldsymbol{\alpha}_p/R) \lra \mathrm{End}(\mathrm{Lie}(\boldsymbol{\alpha}_p/R)) = R
\]
is the identity. It remains to show that the second map is injective. Let $u \in \mathrm{End}(\boldsymbol{\alpha}_p/R)$ act trivially on $\mathrm{Lie}(\boldsymbol{\alpha}_p/R)$. Then $u$ corresponds to an element $f$ of $t^2R[t]/(t^p)$ such that $f(x+y)=f(x)+f(y)$. Expanding $f(t)=\sum_{j=2}^{p-1}{c_jt^j}$ with $c_j \in R$ shows that $f=0$, i.e. that $u=0$. 
\end{proof}
\end{lemma}

Our geometric study of non-split Cartan moduli problems relies on a study of the endomorphism rings of the group scheme $E[p^n]$. It is a classical result, see e.g. \cite[Lemma 4.1.4]{Stu25}, that the functors 
\begin{equation}
\label{eqn:functor-end}
\underline{\mathrm{End}}_k(E[p^n]) \ : \  T \longmapsto \mathrm{End}_T(E[p^n]_T) \qquad \qquad 
\underline{\mathrm{Aut}}_k(E[p^n]) \ : \ T \longmapsto \mathrm{Aut}_T(E[p^n]_T) 
\end{equation}
from $k$-schemes to sets are representable by an affine ring scheme $E_n$ (resp. group scheme $E_n^{\times}$) of finite type over $k$. Likewise, the functor $E_0 = \underline{\mathrm{End}}_k(\boldsymbol{\alpha}_p)$ of group scheme endomorphisms of $\boldsymbol{\alpha}_p$ is a ring scheme. By \cref{lemma:supersing-lie}, the Lie algebra action yields a canonical isomorphism $E_0 \simeq \mathbb{A}^1_k$ of ring schemes.  
\begin{prop}
\label{prop:supersingular-endomorphisms-modp-smooth}
The scheme $E_n$ is smooth of relative dimension $1$ over $k$. The morphism of group schemes
\[
\nu : E_0 = \underline{\mathrm{End}}_k(\boldsymbol{\alpha}_p) \ \lra \ E_n = \underline{\mathrm{End}}_k(E[p^n])
\]
induced by the sequence \cref{eqn:alpha_p} is an open immersion onto the unit component of $E_n$.
\begin{proof}
Since the maps $\pi$ and $\iota$ from \eqref{eqn:ext-Ep} are faithfully flat, resp. a closed immersion, $\nu$ is a monomorphism of $k$-group schemes. Therefore, $E_n$ is equidimensional of positive dimension by \cite[Exp. $\mathrm{VI}_{\mathrm{B}}$, 1.5]{SGA3I}.

Let $R$ be a $k$-algebra. Any $u \in \mathrm{End}_R(E[p^n]_R)$ induces an $R$-linear endomorphism $u^{\sharp}$ of the finite free $R$-module $\cO(E[p^n]) \otimes_k R$. The rule $u \mapsto \det(u^{\sharp}) \in R$ defines a morphism of schemes 
\[
\det: E_n \ \lra \ \mathbb{A}^1.
\]
Since $u$ is an isomorphism iff $u^{\sharp}$ is an isomorphism of $R$-modules, one has $E_n^{\times} = \det^{-1}(\mathbb{A}^1\backslash \{0\})$, so $E_n^{\times} \rightarrow E_n$ is an open immersion. By \cite[Exp. $\mathrm{VI}_{\mathrm{B}}$, 1.3, 1.5]{SGA3I} and \cite[Cor. 4.9.3]{Stu25}, $E_n$ is smooth of relative dimension $1$ and $\nu$ is an open immersion. By \cite[Exp. $\mathrm{VI}_{\mathrm{B}}$, 1.4.2]{SGA3I}, since $E_0$ is connected, $E_0 \rightarrow E_n$ is a closed open immersion onto the unit component of $E_n$. 
\end{proof}
\end{prop}

We now determine the ring of connected components of the endomorphism scheme $E_n$. A \emph{complete} supersingular elliptic curve is a supersingular $E/k$ such that every geometric endomorphism of $E$ is defined over $k$. Given such an $E/k$, let $\OO_E := \mathrm{End}(E)$ with maximal two-sided ideal $\mathfrak{P}_E \subset \OO_E$, and set
\begin{equation}
\label{eqn:def-Rn}
R_n(E) := \OO_E/p^{n-1}\mathfrak{P}_E.
\end{equation}

\begin{prop}
\label{prop:supersingular-endomorphisms-modp-components}
Suppose $E/k$ is complete. Then the quotient $E_n/E_0$ is a finite constant $k$-scheme, and the obvious map $\OO_E \rightarrow E_n(k)$ induces an isomorphism
\[
R_n(E) \simeq (E_n/E_0)(k).
\]
\begin{proof}
By \cite[Exp. $\mathrm{VI}_{\mathrm{A}}$, 5.5]{SGA3I}, the quotient $Q := E_n/E_0$ exists and is \'etale over $k$. Since $E_n \rightarrow Q$ is surjective, $Q$ is quasi-compact hence finite \'etale. The claim holds over $\overline{k}$ by a Dieudonn\'e module calculation (in the proof of \cite[Prop 4.9.4]{Stu25}, replace automorphisms with endomorphisms). Because $\mathrm{End}(E) = \mathrm{End}(E_{\overline{k}})$, one has $Q(k)=Q(\overline{k})$, so $Q$ is a constant $k$-scheme and the result follows. 
\end{proof}
\end{prop}

\begin{remark}
    For every supersingular $j \in \overline{\F}_p$, there is a complete $E/\F_{p^2}$ with $j(E)=j$, see \cite[\S 2.3.7]{GL25}. 
\end{remark}

A final ingredient is the compatibility between the above decomposition of $E_n$ and its ring structure. Making this precise requires uniquely specifying the map $\pi$, which we have not done so far. 

\begin{lemma}
\label{lem:supersingular-endomorphisms-modp-ring}
Suppose $E/k$ is complete. Choose the map $\pi : E[p] \to \boldsymbol{\alpha}_p$ such that Frobenius is equal to 
\[
\iota^{(p)} \circ \pi : E[p] \ \lra \ E^{(p)}[p]. 
\]
Let $u,v \in \OO_E$ and let $u_0,v_0 \in k$ be their actions on $\mathrm{Lie}(E[p^n]/k)$. Let $S$ be a $k$-scheme, then
\[
(u+\nu( \beta))(v+\nu( \gamma )) = uv+\nu (u_0\gamma + v_0^p\beta), \qquad \qquad {\forall} \ \beta,\gamma \in \mathrm{End}(\boldsymbol{\alpha}_p /S) \simeq \mathbb{A}^1(S)
\]
\begin{proof}
    Since $u \circ \iota = \iota \circ u_0$, it suffices to show that $\pi \circ v=v_0^p\circ \pi$, so we may assume $n=1$. We have
\begin{align*}
\iota^{(p)} \circ \pi \circ v &= \mathrm{Frob}_{E[p]} \circ v = v^{(p)} \circ \mathrm{Frob}_{E[p]} = v^{(p)} \circ \iota^{(p)} \circ \pi = (v \circ \iota)^{(p)} \circ \pi \\
&= (\iota \circ v_0)^{(p)} \circ \pi = \iota^{(p)} \circ v_0^p \circ \pi,
\end{align*}
which implies the conclusion since $\iota^{(p)}$ is a closed immersion. 
\end{proof}

\end{lemma}

\subsection{The moduli problems $\cC_{\ns}(p^n)$ and $\cC_{\ns}^+(p^n)$}
\label{subsec:moduli-probems}
We introduce non-split Cartan moduli problems, and prove their representability. Henceforth, $\Q_{p^2}$ denotes the quadratic unramified extension of $\Q_p$, with ring of integers $\Z_{p^2}$ and Frobenius automorphism $\mathrm{Frob}$. Fix a $\Z_p$-basis $(1,z)$ of $\Z_{p^2}$, and let 
\begin{equation}
\label{eqn:prim-elt}
f(t) = t^2+at+b \in \Z_p[t]
\end{equation}
be the characteristic polynomial of $z$. 

\begin{definition}
\label{def:Cns}
For any elliptic curve $E$ over a scheme $S$, we let $\cC_{\ns}(p^n)(E/S)$ be the set of ring homomorphisms $\Z_{p^2}/(p^n) \rightarrow \mathrm{End}_{S}(E[p^n])$. With the obvious action on morphisms, this defines a functor 
\[
\cC_{\ns}(p^n) : \underline{\mathbf{Ell}}_{\Z} \ \lra \ \underline{\mathbf{Set}}.
\]
\end{definition}

\begin{remark}
If $R$ is any ring, assumed unital but not necessarily commutative, and $\alpha: \Z_{p^2}/(p^n) \rightarrow R$ is a ring homomorphism, its kernel is of the form $p^k\Z_{p^2}/(p^n)$ for some $1 \leq k \leq n$. If $p^{n-1} \neq 0$ in $R$, then $\alpha$ is injective and $\alpha \neq\alpha \circ \mathrm{Frob}$. We will repeatedly use this fact in the following two cases: 
\begin{itemize}[label=\tiny$\bullet$]
    \item $R=\mathrm{End}(E[p^n])$ for some elliptic curve $E/S$,
    \item $R=\cO/I$, where $\cO$ is a quaternion order and $I$ is a two-sided ideal with $p^{n-1} \notin I$. 
\end{itemize} 
\end{remark}

\begin{prop}
\label{prop:nonsplit-basic}
The moduli problem $\cC_{\ns}(p^n)$ is relatively representable and affine of finite presentation. It is finite \'etale above $\Z[1/p]$: for any elliptic curve $E/S/\Z[1/p]$ the map $\cC_{\ns}(p^n)_{E/S} \to S$ is finite \'etale. 
\begin{proof}
For any scheme $S$, there is an obvious bijection between $\cC_{\ns}(p^n)(E/S)$ and the set 
\[
\{u \in \mathrm{End}_{S}(E[p^n]),\,u^2+a \cdot u+b\cdot \mathrm{id} = 0\}.
\]
The functor $T \mapsto \mathrm{End}_T(E_T[p^n])$ is representable by an affine morphism $X \rightarrow S$ of finite presentation \cite[Lemma 4.1.4]{Stu25}, hence $\cC_{\ns}(p^n)$ is relatively representable and affine of finite presentation.

Let $S$ be a scheme, and $G$ be the constant $S$-group scheme of $(\Z/p^n\!\Z)^{\oplus 2}$. The functor sending $T/S$ to the set of morphisms $\Z_{p^2}\!/(p^n) \rightarrow \mathrm{End}(G_T)$ is represented by the constant $S$-scheme attached to 
\[
\{M \in M_2(\Z/p^n\Z),\,M^2+aM+bI_2=0\}.
\]
Assume now that $S$ is a $\Z[1/p]$-scheme and let $E/S$ be an elliptic curve. \'Etale-locally over $S$, the group scheme $E[p^n]$ is isomorphic to the constant group scheme with underlying set $(\Z/p^n\!\Z)^{\oplus 2}$. Hence the $S$-scheme $\cC_{\ns}(p^n)_{E/S}$ is \'etale-locally (over $S$) constant, thus the structure map is finite \'etale. 
\end{proof}

\end{prop}

Over $\Z[1/p]$ \cref{def:Cns} is the classical notion of non-split Cartan structure. For $N\geq 1$, write $[\Gamma(N)]$ for the set of \emph{Drinfeld bases} of the $N$-torsion of an elliptic curve (rather than the $\Gamma(N)$-structure of \cite[(3.1)]{KM85}). Then $\mathrm{GL}_2(\Z/N\!\Z)$ acts on the \emph{left} on $[\Gamma(N)]$: for a Drinfeld basis $P,Q \in E[N](S)$ we have
\[
\begin{pmatrix}a & b\\c & d\end{pmatrix} \cdot (P,Q) = (aP+bQ,cP+dQ), \qquad \qquad \begin{pmatrix}a & b\\c & d\end{pmatrix} \in \GL_2(\Z/N\!\Z). 
\]

\begin{lemma}
\label{lemma:recovers-classical-nonsplit}
Fix a ring homomorphism $\iota: \Z_{p^2}/(p^n) \rightarrow M_2(\Z/p^n\Z)$, and let $\Gamma_{\ns}(p^n)\leq \mathrm{GL}_2(\Z/p^n\Z)$ be the image of $(\Z_{p^2}/(p^n))^{\times}$. The morphism $\pi: [\Gamma(p^n)] \rightarrow \cC_{\ns}(p^n)$ of moduli problems over $\Z[1/p]$, defined by 
\[
(P,Q) \in [\Gamma(p^n)](E/S)\ \  \longmapsto \ \ \left[z \mapsto \left(\!\begin{pmatrix}P\\Q\end{pmatrix} \mapsto \iota(z)\begin{pmatrix}P\\Q\end{pmatrix}\!\right)\right] \in \cC_{\ns}(p^n)(E/S) \] 
factors through an isomorphism 
\[
[\Gamma(p^n)]/\Gamma_{\ns}(p^n) \stackrel{\sim}{\lra} \cC_{\ns}(p^n).
\]

\begin{proof}
Consider $E/S/\Z[1/p]$ and $(P,Q) \in \Gamma(p^n)(E/S)$. By \cite[Prop. 1.10.2 (3)]{KM85}, $(a,b) \mapsto aP+bQ$ is an isomorphism between the constant $S$-group scheme attached to $(\Z/p^n\Z)^{\oplus 2}$ and $E[p^n]$, which implies that $\pi$ is well-defined. Since $\pi$ is invariant under the action of $\Gamma_{\ns}(p^n)$, it factors through a morphism $\tilde{\pi}: [\Gamma(p^n)]/\Gamma_{\ns}(p^n) \rightarrow \cC_{\ns}(p^n)$ by \cite[Theorem 7.1.3]{KM85}, where both sides are finite \'etale relatively representable. To check that $\tilde{\pi}$ is an isomorphism, it is enough by \cite[Corollaire 17.9.5]{EGAIV4} to check that, for any elliptic curve $E$ over an algebraically closed field $k$ of characteristic prime to $p$, 
\[
\pi(E/k): [\Gamma(p^n)](E/k)/\Gamma_{\ns}(p^n) \rightarrow \cC_{\ns}(p^n)(E/k)
\]
is a bijection, which is straightforward. 
\end{proof}
\end{lemma}

We now turn to the moduli problem attached to \emph{normalisers} of non-split Cartan subgroups.

\begin{definition}
The group $\Z/2\!\Z$ acts on $\cC_{\ns}(p^n)$ by pre-composing the given ring homomorphism by the Frobenius automorphism of $\Z_{p^2}/(p^n)$. The moduli problem $\cC_{\ns}^+(p^n)$ is the quotient of $\cC_{\ns}(p^n)$ by $\Z/2\!\Z$. 
\end{definition}

\begin{prop}
\label{prop:nonsplit-plus-basic}
The moduli problem $\cC_{\ns}^+(p^n)$ is well-defined, relatively representable and affine of finite presentation. For any elliptic curve $E/S$, $\cC_{\ns}(p^n)_{E/S}$ is an \'etale $\Z/2\!\Z$-torsor over $\cC_{\ns}^+(p^n)_{E/S}$ through the forgetful map, which induces an isomorphism 
\[
\cC_{\ns}(p^n)_{E/S}\ /\ (\Z/2\!\Z) \ \ \stackrel{\sim}{\lra} \ \ \cC_{\ns}^+(p^n)_{E/S}.
\]
Furthermore, $\cC_{\ns}^+(p^n)$ is finite \'etale over $\Z[1/p]$. The map of \cref{lemma:recovers-classical-nonsplit} induces an isomorphism 
\[
[\Gamma(p^n)]\,/\, \Gamma_{\ns}^+(p^n) \ \ \stackrel{\sim}{\lra} \ \ \cC_{\ns}^+(p^n).
\]

\begin{proof}
Let $E/S$ be an elliptic curve and $\alpha: \Z_{p^2}/(p^n) \rightarrow \mathrm{End}_S(E[p^n])$ be a ring homomorphism: as we saw, $\alpha \neq \alpha \circ \mathrm{Frob}$, so $\Z/2\!\Z$ acts freely on $\cC_{\ns}(p^n)$. Since $\cC_{\ns}(p^n)$ is relatively representable affine of finite presentation, the result follows from \cite[Theorem 7.1.3]{KM85}. 
\end{proof}

\end{prop}

The most interesting features of the moduli problems $\cC_{\ns}(p)$ and $\cC_{\ns}^+(p)$ appear in characteristic $p$. The key to the smoothness results of \cref{subsec:smoothness} is the following. 

\begin{lemma}
\label{lem:hasse-vanishes}
Let $E$ be an elliptic curve over a ring $R$ of characteristic $p$. Suppose that $\cC_{\ns}(p^n)(E/R)$ is not empty. Then the Hasse invariant of $E$ vanishes.\footnote{This is stronger than any base change of $E$ to a field being supersingular, which holds if and only if the Hasse invariant is nilpotent.}

\begin{proof}
We may assume that $n=1$, so that there is a ring homomorphism $\iota: \F_{p^2} \rightarrow \mathrm{End}(E[p])$. Working Zariski-locally over $\operatorname{Spec}(R)$, by \cite[(12.4.1.3)]{KM85} and \cite[Exp. $\mathrm{VII}_{\mathrm{A}}$, \S 6.1]{SGA3I}, it is enough to show that the $p$-Lie algebra structure on $\mathrm{Lie}(E/R)$ is trivial. By \cref{lemma:supersing-lie} and \cite[Exp. $\mathrm{VII}_{\mathrm{A}}$, \S 6.1.1]{SGA3I}, it suffices to show that the structure of $p$-Lie algebra on $\mathrm{Lie}(E[p]/R)$ is trivial.

Let $\alpha \in \F_{p^2}$ be the reduction of $z$, as in \eqref{eqn:prim-elt}, then $\iota(\alpha)$ is an endomorphism of $\mathrm{Lie}(E[p]/R) \simeq R$, it acts by multiplication by some $\zeta \in R$ such that $\zeta^2+a\zeta+b=0$: in particular, $u+v\alpha \in \F_{p^2} \mapsto u+v\zeta \in R$ is a ring homomorphism. Now, let $x \in \mathrm{Lie}(E[p]/R)$ be a generator, then one has \[\zeta x^{(p)} = \iota(\alpha) x^{(p)} = (\iota(\alpha) x)^{(p)}=(\zeta x)^{(p)}=\zeta^p x^{(p)}.\]
Because $\alpha-\alpha^p \in \F_{p^2}^{\times}$, one has $\zeta-\zeta^p \in R^{\times}$, so $x^{(p)}=0$ and the conclusion follows. 
\end{proof}
\end{lemma}

\begin{corollary}
\label{cor:factors-through-supersingular-locus}
    Let $\cP$ be the moduli problem $\cC_{\ns}(p^n)$ or $\cC_{\ns}^+(p^n)$ over $\F_p$. For any $E/S/\F_p$ let $S_0 \subset S$ be the vanishing locus of the Hasse invariant. Then the obvious map of $S$-schemes is an isomorphism: 
    \[\cP_{E_{S_0}/S_0} \ \stackrel{\sim}{\lra} \ \cP_{E/S}.
    \]
    \begin{proof}
        The proof of \cref{lem:hasse-vanishes} shows $S_0$ is well-defined. Working \'etale-locally over $S$, by \cref{prop:nonsplit-plus-basic}, we may assume $S$ is affine and $\cP = \cC_{\ns}(p^n)_{\F_p}$. 
        The conclusion follows from \cref{lem:hasse-vanishes}.
    \end{proof}
\end{corollary}

\begin{prop}
\label{prop:restrict-to-exactly-supersingular}
Let $\cP$ denote the moduli problem $\cC_{\ns}(p^n)$ or $\cC_{\ns}^+(p^n)$ over a perfect field $k$ of characteristic $p$. Let $\mathcal{L}$ be a finite \'etale representable moduli problem over $k$ with fine moduli space $\cM$ and universal elliptic curve $\cE$. Let $H \subset \cM$ be the reduced closed subscheme corresponding to the finite set of closed points with supersingular $j$-invariants, and $E/H$ be the pull-back of $\cE/\cM(\cL)$. Let $q: \cP_{E/H} \rightarrow H$ be the structure morphism. Then the elliptic curve $(q^{\ast}E)/\cP_{E/H}$ represents the moduli problem $\cP \times \cL$. 

\begin{proof}
Let $\tilde{q}: \cP_{\cE/\cM} \rightarrow \cM$, then $(\tilde{q}^{\ast}\cE)/\cP_{\cE/\cM}$ represents $\cP \times \cL$ by \cite[(4.3.4)]{KM85}. By \cref{cor:factors-through-supersingular-locus}, we only need to show that $H$ is the vanishing locus $Z$ of the Hasse invariant in $\cM$. It is enough to check that $H$ and $Z$ are both finite and geometrically reduced $k$-schemes with the same $\overline{k}$-points. By construction, $H$ is finite reduced; $k$ is perfect so it is geometrically reduced. By \cite[Theorem 12.4.3]{KM85}, $Z(\overline{k})$ consists exactly of points with supersingular $j$-invariant, so $Z$ is also finite. Since $\cL$ is \'etale, \emph{loc.cit.} also implies that $Z$ is geometrically reduced. 
\end{proof}

\end{prop}

This result admits the following more concrete formulation over $k=\overline{\F}_p$. 

\begin{corollary}
\label{cor:finer-description-supersing}
Let $\cP$ denote the moduli problem $\cC_{\ns}(p^n)$ or $\cC_{\ns}^+(p^n)$ over $k := \overline{\F}_p$. Let $\mathcal{L}$ be a finite \'etale representable moduli problem over $k$. Let $\cM(\cP,\cL)$ be the fine moduli scheme attached to the moduli problem $\cP \times \cL$. Then the classifying map (cf. \cite[(8.1.3)]{KM85}) defines an isomorphism 
\[f: \coprod_{(E,z) \in \mathcal{L}^{\mathrm{ss}}}{\cP_{E/k}} \ \stackrel{\sim}{\lra} \ \cM(\cP,\cL),\] where $\mathcal{L}^{\mathrm{ss}}$ is the finite set of isomorphism classes of $(E,z)$, where $E/k$ is supersingular and $z \in \mathcal{L}(E/k)$. 
\end{corollary}

It remains to understand the scheme $\cC_{\ns}(p^n)_{E/k}$ where $E/k$ is a complete supersingular curve. Define an action of $\underline{\mathrm{End}}(\boldsymbol{\alpha}_p)$ on $\cC_{\ns}(p^n)_{E/k}$ as follows: for any $k$-algebra $R$, there is a canonical bijection between $\cC_{\ns}(p^n)(E/R)$ and $\{\phi \in \mathrm{End}(E[p^n]_R) : f(\phi) = 0\}$ with $f$ as in \eqref{eqn:prim-elt}. For any $\lambda \in \mathrm{End}(\boldsymbol{\alpha}_p/R)$ set 
\[
\lambda \star \phi := \phi + \nu(\lambda) \quad \in \mathrm{End}(E[p^n]_R). 
\]
Note that $f(\lambda \star \phi) = 0$: By \cref{lem:supersingular-endomorphisms-modp-ring,prop:supersingular-endomorphisms-modp-components} it suffices to show that $\phi_0 + \phi_0^p \equiv -a \pmod{p}$ for any $\phi \in \mathrm{End}(E)$ satisfying $f(\phi) = 0$. This holds since $f(\phi_0)=0$ by \cref{lemma:supersing-lie} so that $\phi_0 \in \F_{p^2}$ and 
\begin{equation}
\label{eqn:polyfac}
t^2+at+b=(t-\phi_0)(t-\phi_0^p) \quad \mbox{in} \ \ k[t]. 
\end{equation}
The following is a consequence of the results above.

\begin{prop}
\label{prop:description-ccns}
Let $E/k$ be a complete supersingular curve. There is an $\underline{\mathrm{End}}(\boldsymbol{\alpha}_p)$-equivariant isomorphism
\[\rho : \cC_{\ns}(p^n)_{E/k} \ \ \stackrel{\sim}{\lra} \ \ \coprod_{u: \Z_{p^2}/(p^n) \rightarrow R_n(E)}{\underline{\mathrm{End}}(\boldsymbol{\alpha}_p)}\] sending a ring homomorphism $u: \Z_{p^2}/(p^n) \rightarrow \mathrm{End}(E[p^n]_{\overline{k}})$ to the component indexed by the projection of $u$ modulo the unit component of $\underline{\mathrm{End}}_k(E[p^n])$. The action of $\Z/2\!\Z$ satisfies:
\begin{itemize}[noitemsep,label=\tiny$\bullet$]
    \item it acts on the index set by pre-composing with Frobenius, so that it has no fixed point on $\pi_0(\cC_{\ns}(p^n)_{E/k})$.  
    \item it \emph{reverses} the $\underline{\mathrm{End}}(\boldsymbol{\alpha}_p)$-action: for $\lambda \in \underline{\mathrm{End}}(\boldsymbol{\alpha}_p)(\overline{k})$ and $u \in \cC_{\ns}(p^n)_{E/k}(\overline{k})$, one has 
\[\overline{1} \cdot (\lambda \star u) = (-\lambda) \star (\overline{1}\cdot u).\] 
\end{itemize}
In particular, $\cC_{\ns}^+(p^n)_{E/k}$ is isomorphic to the disjoint union of $p^{2n-2}$ copies of $\mathbb{A}^1_k$. 
\begin{proof}
This follows from \cref{prop:supersingular-endomorphisms-modp-smooth,prop:supersingular-endomorphisms-modp-components}. For the final claim, restrict $\rho^{-1}$ to a subset of components corresponding to a set of representatives for morphisms $\Z_{p^2}/(p^n) \rightarrow R_n(E)$ modulo pre-composition by Frobenius, and use \cref{prop:nonsplit-plus-basic}. For the component count, one can check that $R_n(E)^{\times}$ acts transitively on its subrings $R \simeq \Z_{p^2}/(p^n)$ and that the stabiliser of a given $R$ is $R^{\times}$, and then apply \cite[Lemma 26.6.7]{Voi21} to the preimage of $R$ in $\OO_E$.
\end{proof}
\end{prop}

\begin{corollary}
\label{cor:smoothness-level-modp}
Let $k$ be a field of characteristic $p$ and $\cL$ a finite \'etale representable moduli problem over $k$. Then the moduli problem $\cC_{\ns}(p^n)\times \cL$ (resp. $\cC_{\ns}^+(p^n)\times \cL$) is representable by a smooth affine scheme of relative dimension one. If $k$ is algebraically closed, this scheme is the disjoint union of copies of $\mathbb{A}^1_k$ indexed by the $k$-isomorphism classes of $(E,u,z)$ (resp. $(E,R,k)$), where $E/k$ is supersingular, $u: \Z_{p^2}/(p^n) \rightarrow R_n(E)$ is a ring homomorphism (resp. $R \subset R_n(E)$ a subring isomorphic to $\Z_{p^2}/(p^n)$), and $z \in \cL(E/k)$. 

\begin{proof}
Since smoothness of relative dimension one can be tested fpqc-locally on the base, we may assume that $k$ is algebraically closed. The claim then follows from the previous results. 
\end{proof}

\end{corollary}

\begin{prop}
\label{prop:smoothness-level}
Let $R$ be any ring and $\cL$ be a finite \'etale representable moduli problem over $R$. Then the moduli problem $\cC_{\ns}(p^n)\times \cL$ (resp. $\cC_{\ns}^+(p^n)\times \cL$) is representable by a $R$-scheme which is affine and smooth of relative dimension one over $R$.  

\begin{proof}

By \cite[(4.12)]{KM85} we may assume that $\cL = [\Gamma_1(N)]$ over $R=\Z[1/N]$ for $N \geq 4$. By previous results and \cite[Prop. 1.3.4]{Stu25}, since smoothness can be tested fpqc-locally on the base, we may even assume that $\cL=[\Gamma_1(N)]$ over $R=W(\overline{\F}_p)$ for $N \geq 4$ prime to $p$. The fine moduli scheme $\cM$ attached to $\cC_{\ns}(p^n)\times \cL$ (resp. $\cC_{\ns}^+(p^n)\times \cL$) exists and is affine; we wish to show that it is smooth of relative dimension one over $R$. This holds for its generic fibre (the moduli problem is finite \'etale representable), which is irreducible (by considering the complex uniformization). It also holds for the special fibre by \cref{cor:smoothness-level-modp}. By \cite[Prop. 4.3.4]{Stu25}, it is enough to show that every connected component of the special fibre of $\cM$ contains the specialisation of some $R$-point of $\cM$.

By our description of the special fibre of $\cM$ and since $\cL$ is \'etale, it is enough to show the following: for every supersingular elliptic curve $E_0/\overline{\F}_p$ and every ring homomorphism $u: \Z_{p^2}/(p^n) \rightarrow R_n(E_0)$, there exists an elliptic curve $E/R$, an isomorphism $g: E_0 \rightarrow E_{\overline{\F}_p}$ such that 
\[
\mathrm{End}(E) \ \overset{g}{\lra} \ \mathrm{End}(E_0) \ \lra \ R_n(E_0)
\]
has the same image as $u$. This can be shown by a direct adaptation of \cite[Proposition 2.7]{GZ85}.
\end{proof}
\end{prop}

\subsection{Smoothness of the coarse moduli scheme}
\label{subsec:smoothness}

We now turn to removing the auxiliary level structure, and study the coarse moduli scheme associated to $\cC_{\ns}(p^n)$ and $\cC_{\ns}^+(p^n)$. 

\begin{lemma}
    \label{lem:action-by-auto}
Let $E/k$ be a complete supersingular elliptic curve, $\phi \in \mathrm{Aut}(E/k) \backslash \{ \pm 1\}$, and define
\[
Y := \cC_{\ns}(p^n)_{E/k} \qquad \qquad Y^+ := \cC_{\ns}^+(p^n)_{E/k}.
\]
Then $\phi$ has finitely many fixed points in $Y(\overline{k})$ and $Y^+(\overline{k})$. Furthermore, every element of $\pi_0(Y^+\!/\overline{k})$, resp. $Y^+(\overline{k})$, that is fixed by $\phi$ is the image of an element of $\pi_0(Y/\overline{k})$, resp. $Y(\overline{k})$, fixed by $\phi$. 
\begin{proof}
Fix a $\Z_p$-basis $(1,z)$ of $\Z_{p^2}$, then $Y$ is endowed with an action of $\underline{\mathrm{End}}(\boldsymbol{\alpha}_p)$, and the connected components of $Y$ and $Y^+$ are geometrically connected, see \cref{prop:description-ccns}. 

Let $\phi_0 \in k$ be the action of $\phi$ on $\mathrm{Lie}(E/k)$. Since $\phi\phi^{\vee}=1$ it follows from \eqref{eqn:polyfac} that $\phi_0^{p+1} = 1$. Let $u: \Z_{p^2}/(p^n) \rightarrow R_n(E)$ be a ring homomorphism such that $\phi u\phi^{-1} = u \circ \mathrm{Frob}^r$ with $r \in \{0,1\}$. Since $\OO_E/\mathfrak{P}_E$ is commutative, $u=\phi u\phi^{-1}$ modulo $\mathfrak{P}_E$, and in particular $u(z-\mathrm{Frob}^r(z)) \in \mathfrak{P}_E$. Hence $z-\mathrm{Frob}^r(z) \notin \Z_{p^2}^{\times}$, thus $r=0$. Hence, if $x \in \pi_0(Y^+)$ is fixed by $\phi$, then $x$ is the image of $y \in \pi_0(Y)$ fixed by $\phi$. Likewise, if $x \in Y^+(\overline{k})$ is fixed by $\phi$, then it is the image of some $y \in Y(\overline{k})$ fixed by $\phi$. 

Let $u: \Z_{p^2}/(p^n) \rightarrow \OO_E/(p^n) \subset \mathrm{End}(E[p^n])$ be a ring homomorphism that commutes with $\phi$ modulo $p^{n-1}\mathfrak{P}_E$, so the connected component of $u$ in $Y$ is stable under $\phi$. For any $\lambda \in \mathrm{End}(\boldsymbol{\alpha}_p/\overline{k})$ we have
\[
\phi \cdot (\lambda \star u) = \phi_0^2\lambda \star \phi u\phi^{-1}.
\]
If $p$ is inert in the order $\Z[\phi]$, then $\phi_0^2 \neq 1$, so $\phi$ has exactly one fixed point in the connected component of $u$ in $Y$.  Otherwise, either $p=2$ and $\phi$ has order $4$, or $p=3$ and $\phi$ has order $3$ or $6$. In either case $\phi_0^2=1$, so $\phi$ acts on the connected component of $u$ as a translation. In particular, if $\phi u\phi^{-1} \neq u$, then $\phi$ has no fixed point on the connected component of $u$. Suppose $\phi u\phi^{-1} = u$, then a calculation in $\OO_E/(p)$ shows that $u(z)-c$ is nilpotent for some $c \in \F_p$. This is impossible since $(1,z)$ is a $\Z_p$-basis of $\Z_{p^2}$. 
\end{proof}
\end{lemma}

The following results conclude the proof of \cref{thm:A}.
Let us say that a moduli problem $\cL$ over a ring $R$ is \emph{symmetric} if, for every elliptic curve $E/S/R$, the automorphism $-1$ of $E$ acts trivially on $\cL(E/R)$. Note that $\cC_{\ns}(p^n)$ and $\cC_{\ns}^+(p^n)$ are symmetric for any $n \geq 1$, as is $[\Gamma(N)]/\Gamma$ for any $\Gamma \leq \mathrm{GL}_2(\Z/N\Z)$ containing $\pm I_2$. Any finite product of symmetric moduli problems is also symmetric.

\begin{prop}
\label{prop:smoothness-coarse}
Let $R$ be a regular ring, $\cL$ a finite \'etale symmetric relatively representable moduli problem over $R$, and $\cP$ the moduli problem $\cC_{\ns}(p^n)$ or $\cC_{\ns}^+(p^n)$ over $R$. Then $\cP \times \cL$ has a coarse moduli scheme $\mathrm{M}(\cP \times \cL)$ over $R$, which is smooth of relative dimension $1$ and satisfies coarse base change to regular rings.
\begin{proof}

The coarse moduli scheme $M=\mathrm{M}(\cP \times \cL)$ exists by \cite[(8.1.1)]{KM85}, since $\cP \times \cL$ is relatively representable and affine. To check the other claims, we may work Zariski-locally on $R$ and assume that some integer $N \geq 3$ is invertible on $R$. Then the moduli scheme $\cM := \cM(\cP \times \cL \times [\Gamma(N)])$ is a smooth affine $R$-scheme of relative dimension one by \cref{prop:smoothness-level}. It carries an action of the group $G := \mathrm{GL}_2(\Z/N\Z)/\{\pm \mathrm{id}\}$ because $\cL$ is symmetric, and one has $M=\cM/G$. 

The smoothness of $M$ follows from the first theorem in the Notes of Chapters 8 and 10 (p. 508) of \cite{KM85}. To check coarse base change along morphisms of regular rings, we wish to apply the theorem on p. 510 of \cite{KM85}: one has to check that, for every non-unit $g \in G$, every geometric fibre of $\cM \rightarrow \operatorname{Spec}{R}$ has finitely many fixed points under $g$.

It is enough to show that, if $k$ is an algebraically closed field, there are finitely many isomorphism classes of $(E,\alpha)$ (where $E/k$ is an elliptic curve and $\alpha \in \cP(E/k)$) such that $E$ admits an automorphism $j$ (which is not $\pm \mathrm{id}$) preserving $\alpha$. If $p$ is invertible in $k$, $\cP$ is finite \'etale and $\mathrm{Aut}(E)^{\times}=\{\pm 1\}$ unless $j(E) \in \{0,1728\}$, so this is clear. Otherwise, this follows from \cref{lem:action-by-auto}. 
\end{proof}
\end{prop}

\begin{corollary}
\label{cor:special-fibre-of-coarse}

Let $R$ be a DVR with residue field $k$ of characteristic $p$. Let $\cP$ be the moduli problem $\cC_{\ns}(p^n)$, resp. $\cC_{\ns}^+(p^n)$, and $\cL$ a finite \'etale symmetric relatively representable moduli problem over $R$. The geometric special fibre of the coarse moduli scheme $\mathrm{M}(\cP \times \cL)$ is the disjoint union of copies of $\mathbb{A}^1$ indexed by the isomorphism classes of $(E,u,\alpha)$, resp. $(E,A,\alpha)$, where $E/\overline{k}$ is a supersingular elliptic curve, $u: \Z_{p^2}/(p^n) \rightarrow R_n(E)$ is a ring homomorphism, resp. $A \subset R_n(E)$ is a subring isomorphic to $\Z_{p^2}/(p^n)$, and $\alpha \in \cL(E/\overline{k})$.

\begin{proof}
By \cref{prop:smoothness-coarse}, the special fibre (resp. geometric special fibre) of $\mathrm{M}(\cP \times \cL)$ is the coarse moduli scheme attached to the base-changed moduli problem $\cP \times \cL$ over $k$ (resp. $\overline{k}$) and it is smooth of relative dimension one. We may thus assume that $R=k$. 

By \cref{cor:factors-through-supersingular-locus} and \cite[Lem. 8.1.3.1]{KM85}, the set-theoretic image of the $j$-invariant $\mathrm{M}(\cP \times \cL) \rightarrow \mathbb{A}^1_k$ is contained in the finite set $\cS$ of supersingular $j$-invariants; $\mathrm{M}(\cP \times \cL)$ is geometrically reduced, so the scheme-theoretic image of $j$ is the reduced closed subscheme of $\mathbb{A}^1_k$ attached to $\cS$. By \cref{cor:factors-through-supersingular-locus} and \cite[Cor. 12.4.4]{KM85}, if $E/\overline{k}$ is supersingular with universal formal deformation $\cE/\overline{k}\llbracket t \rrbracket$, then $(\cP \times \cL)_{\cE/\overline{k}\llbracket t \rrbracket}$ is isomorphic over $\overline{k}\llbracket t\rrbracket$ to $(\cP \times \cL)_{E/\overline{k}}$, where $\overline{k}$ is viewed as $\overline{k}\llbracket t\rrbracket /(t)$, compatible with the action of $\mathrm{Aut}(E)$. Therefore, by \cite[Prop 8.2.3]{KM85}, the classifying map induces an isomorphism 
\[
\coprod_{\substack{E/\overline{k}\\\alpha \in \cL(E/\overline{k})'}}{\frac{\cP_{E/\overline{k}}}{\mathrm{Aut}(E,\alpha)/\{\pm \mathrm{id}\}}} \ \ \stackrel{\sim}{\lra} \ \  \mathrm{M}(\cP \times \cL)_{\overline{k}},
\] 
where $\cL(E/\overline{k})'$ denotes a system of representatives for $\cL(E/\overline{k})$ modulo $\mathrm{Aut}(E)$. Each $\cP_{E/\overline{k}}$ is a union of copies of $\mathbb{A}^1_{\overline{k}}$ indexed by morphisms $\Z_{p^2}/(p^n) \rightarrow R_n(E)$ (resp. subrings of $R_n(E)$ isomorphic to $\Z_{p^2}/(p^n)$), and this identification is compatible with $\mathrm{Aut}(E)$. The result follows, since over any field the quotient of $\mathbb{A}^1$ by a finite group is isomorphic to $\mathbb{A}^1$, see for instance \cite[Lemma 3.2.1]{KM77}.
\end{proof}
\end{corollary}

\begin{remark}\label{rmk:component count}
    The number of components of $\cX_{\ns}^+(p^n)_{\overline{\F}_p}$ can be found by considering, for each supersingular elliptic curve $E/\overline{\F}_p$, the action of $\Aut(E)$ on the $p^{2n-2}$ components of $\cC_{\ns}^+(p^n)_{E/k}$ (\cref{prop:description-ccns}).
\end{remark}

\begin{corollary}
\label{cor:special-fibre-of-xnsp}
The coarse moduli scheme $\cX_{\ns}^+(p)$ attached to the moduli problem $\cC_{\ns}^+(p)$ over $\Z$ exists. It is smooth affine, and has a $j$-invariant map to $\mathbb{A}^1_{\Z}$. The fibre $\cX_{\ns}^+(p)_{\F_p}$ is the disjoint union of the $j^{-1}(s) \simeq \mathbb{A}^1_{\kappa(s)}$, where $s \in \mathbb{A}^1_{\F_p}$ runs over the supersingular points. 
\begin{proof}

The $j$-invariant $j: \cX_{\ns}^+(p)_{\F_p} \rightarrow \mathbb{A}^1_{\F_p}$ has scheme-theoretic image equal to the reduced subscheme of the supersingular locus (because the source is smooth). We need to show that for every supersingular $s \in \mathbb{A}^1_{\F_p}$, the fibre $X_s := j^{-1}(s)$ is isomorphic to $\mathbb{A}^1_{\kappa(s)}$. By \cref{cor:special-fibre-of-coarse} this holds geometrically, so the result follows since all forms of $\mathbb{A}^1_k$ over a perfect field $k$ are trivial \cite[Lemma 1.1]{Rus70}. 
\end{proof}
\end{corollary}

\subsection{Moduli scheme for mixed Cartan structures}
\label{subsec:mixed-moduli}

We now introduce mixed Cartan problems and the associated moduli schemes. They are defined prime by prime, so all that remains to do is introduce the moduli problem attached to the normaliser of a split Cartan subgroup. 

\begin{definition}
    For any prime power $p^n$, define the moduli problems $\cC_{\spl}(p^n)$ and $\cC_{\spl}^+(p^n)$ over $\Z$ as the quotient of $[\Gamma(p^n)]_{\Z}$ by the subgroups $\Gamma_{\spl}(p^n)$ and $\Gamma_{\spl}^+(p^n)$ of diagonal matrices, resp. diagonal or anti-diagonal matrices, in $\mathrm{GL}_2(\Z/p^n\Z)$. Given two coprime positive integers $N_{\spl}, N_{\ns}$, we define
        \[
        \cC(N_{\spl},N_{\ns}) := \prod_{p^k \parallel N_{\spl}}{\!\!\cC_{\spl}(p^k)} \times \prod_{p^k \parallel N_{\ns}}{\!\!\cC_{\ns}(p^k)},\qquad \cC^+(N_{\spl},N_{\ns}) := \prod_{p^k \parallel N_{\spl}}{\!\!\cC_{\spl}^+(p^k)} \times \prod_{p^k \parallel N_{\ns}}{\!\!\cC_{\ns}^+(p^k)}\]
\end{definition}

\begin{prop}
\label{prop:mixed-cartan-coarse}
    The moduli problems $\cC(N_{\spl},N_{\ns})$ and $\cC^+(N_{\spl},N_{\ns})$ are relatively representable, affine, smooth above $\Z[1/N_{\spl}]$, finite above $\Z[1/N_{\ns}]$, and admit coarse moduli schemes 
    \[
    \cX_{\Car}(N_{\spl},N_{\ns}), \qquad  \cX_{\Car}^+(N_{\spl},N_{\ns})
    \]
    over $\Z$. These schemes are affine, normal, of finite type and flat of relative dimension one over $\Z$. They are smooth over $\Z[1/N_{\spl}]$ and their formation commutes with base change to regular $\Z[1/N_{\spl}]$-algebras. 
    
    \begin{proof}
        By \cite[Theorem 5.1.1, Theorem 7.1.3]{KM85}, one sees that $\cC_{\spl}(p^n), \cC_{\spl}^+(p^n)$ are relatively representable, symmetric, finite and normal. They are \'etale above $\Z[1/p]$. The statement follows from our previous results on non-split Cartan moduli problems and results of \cite{KM85}, namely Theorem 5.1.1, Theorem 7.1.3, \S 8.1 and Proposition 8.2.2.  
    \end{proof}
\end{prop}

\begin{remark}
    The curves $\cX_{\Car}(N_{\spl},N_{\ns})_{\Q}$ and $\cX_{\Car}^+(N_{\spl},N_{\ns})_{\Q}$ are the non-cuspidal loci of the Cartan modular curves described in \cite[\S 3]{DLM22}. They are respectively attached to Cartan and Cartan-plus subgroups of $\mathrm{GL}_2(\Z/N_{\spl}N_{\ns}\!\Z)$ in the terminology of \emph{loc.cit.}.  
\end{remark}

\begin{definition}
    Let $N \geq 1$ be an integer and $A$ be an \'etale free $\Z/N\!\Z$-algebra of rank two. Let $N_{\spl}$ be the product of the primes $p \mid N$ at which $A$ splits. Given an elliptic curve $E/S/\Z[1/N_{\spl}]$, we say that a ring homomorphism 
    \[
    f: A \ \lra \ \mathrm{End}_S(E[N])
    \]
    is \emph{regular} if for every $a \in A \backslash \{0\}$ and every $s \in S$, the endomorphism $f(a)_s$ of $E_s[N]$ is non-zero. 
    For any elliptic curve $E/S/\Z[1/N_{\spl}]$, let $\cP_A(E/S)$ be the set of regular homomorphisms $A \rightarrow \mathrm{End}_S(E[N])$. The group $\mathrm{Aut}(A)$ acts (on the right) by pre-composition on the moduli problem $\cP_A$ over $\Z[1/N]$. 
\end{definition}

We also need a variant of \cref{lemma:recovers-classical-nonsplit} for the mixed Cartan moduli problems. 

\begin{lemma}
    \label{lem:recovers-natural-split}
    Let $p$ be a prime and $n \geq 1$. The moduli problem $\cP := \cP_{(\Z/p^n\Z)^{\oplus 2}}$ is relatively representable and finite \'etale. The morphism of moduli problems 
    $\pi: [\Gamma(p^n)] \rightarrow \cP$ defined by \[(P,Q) \in [\Gamma(p^n)](E/S) \mapsto \left[(z_1,z_2) \in (\Z/p^n\Z)^{\oplus 2} \mapsto \left(\begin{pmatrix}P\\Q\end{pmatrix} \mapsto \begin{pmatrix}z_1P\\z_2Q\end{pmatrix}\right)\right] \in \cP(E/S) \] is well-defined and factors through an isomorphism $\cC_{\spl}(p^n) \simeq \cP$. Pre-composing with the involution 
    \[
    (\Z/p^n\!\Z)^{\oplus 2} \ \lra \ \ (\Z/p^n\!\Z)^{\oplus 2} : (x,y) \longmapsto (y,x),
    \]
    defines an action of $\Z/2\!\Z$ on $\cP$. The quotient $\cP^+ := \cP/(\Z/2\!\Z)$ is well-defined over $\Z[1/p]$, relatively representable and finite \'etale. Composing $\pi$ with $\cP \rightarrow \cP^+$ induces an isomorphism $\cC_{\spl}^+(p^n) \simeq \cP^+$.   
    \begin{proof}
        To show the relative representability and that $\cP_{(\Z/p^n\Z)^{\oplus 2}}$ is finite \'etale, it is enough to show by \'etale descent 
        that if $E$ is an elliptic curve over some $\Z[1/p]$-scheme $S$ and $E[p^n]$ is a constant $S$-scheme, then $\cP_{E/S}$ exists and is a finite \'etale $S$-scheme, which is a direct verification. The rest is proved in a similar way to \cref{lemma:recovers-classical-nonsplit} and \cref{prop:nonsplit-plus-basic}.
    \end{proof}
\end{lemma}

Let $N_{\ns},N_{\spl}$ be two coprime positive integers and $N=N_{\ns} \cdot N_{\spl}$.
Let $A$ be the quadratic \'etale $\Z/N\!\Z$-algebra that is split at prime divisors of $N_{\spl}$, and nonsplit at prime divisors of $N_{\ns}$, i.e. 
\begin{equation}\label{eq:quadratic_algebra}
    A = A(N_{\spl},N_{\ns}) := \Bigl( \prod_{p^k \parallel  N_{\rm spl}} (\Z/p^k\Z)^2 \ \Bigr) \times \Bigl( \prod_{p^k \parallel N_{\rm ns}} \Z_{p^2}/p^k\Z_{p^2}\ \Bigr) .
\end{equation}

\begin{corollary}\label{cor:PA moduli problem}
The moduli problem $\cP_A$ is relatively representable, affine of finite presentation, and naturally isomorphic to $\cC(N_{\spl},N_{\ns})_{\Z[1/N_{\spl}]}$. The group $(\Z/2\Z)^t$ acts on $\cP_A$, generated by pairwise-commuting involutions $\sigma_p$ for each $p\mid N$, where $\sigma_p$ acts by the Frobenius involution on $\Z_{p^2}/(p^k)$ if $p\mid N_{\ns}$, and by swapping the elements of the idempotent basis of $(\Z/p^k\Z)^2$ if $p\mid N_{\spl}$. The quotient moduli problem $\cP_A^+:=\cP_A/(\Z/2\Z)^t$ exists, is relatively representable, affine of finite presentation, and is isomorphic to $\cC^+(N_{\spl},N_{\ns})_{\Z[1/N_{\spl}]}$. 

\begin{proof}
    This follows from prime factorization \cite[Propositions 1.8.2, 3.5.1]{KM85} and the following simpler variant of Prop. 7.3.1 of \emph{op.cit.} (which does not apply since $\cC_{\ns}(p^n)$ or $\cC_{\ns}^+(p^n)$ are not flat). Let $\cP_1, \cP_2$ be two moduli problems over $\underline{\mathbf{Ell}}_{\Z[1/N_{\spl}]}$ that are relatively representable affine and $G_1,G_2$ be finite groups acting \emph{freely} on $\cP_1,\cP_2$ respectively. Then the natural map 
    \[
    (\cP_1 \times \cP_2)/(G_1 \times G_2) \ \lra \ (\cP_1/G_1) \times (\cP_2/G_2)
    \]
    is an isomorphism. Given the construction of the quotient moduli problem (cf. \cite[pp. 190--191]{KM85}), one reduces this claim to its commutative algebra analogue, which follows from \cite[Cor. A7.1.2]{KM85}. 
\end{proof}

\end{corollary}

\subsection{Minimal regular compactifications of $\cX_{\ns}^+(p)$} 
\label{subsec:EP}

Edixhoven--Parent \cite{EP24} construct a minimal regular model $X_{\mathrm{EP}}(p)$ for the modular curve $X_{\ns}^+(p)$ over $W := W(\overline{\F}_p)$ when $p \geq 11$. We will now compare the coarse moduli scheme $\cX_{\ns}^+(p)$ over $W$ with $X_{\mathrm{EP}}(p)$. There is an obvious difference between the two schemes: $X_{\mathrm{EP}}(p)$ is proper and has non-smooth points over $W$, whereas $\cX_{\ns}^+(p)$ is affine and smooth over $W$. We show that this is essentially the only difference: $\cX_{\ns}^+(p)$ is identified with the smooth locus of $X_{\mathrm{EP}}(p)$. For $p<11$ the modular curve $X_{\ns}^+(p)$ has genus $0$ so there is no unique minimal regular model. In these cases, we instead exhibit $\cX_{\ns}^+(p)$ as an open subscheme of $\mathbb{P}^1_{\Z}$.

The following lemma is well-known to experts, similar arguments appear for instance in Zywina \cite[Prop. 1.13]{Zyw15}, Le Fourn--Lemos \cite[Appendix B]{LFL21}, as well as the recent works of Furio--Lombardo \cite[Chapter 3]{FL25}, Bisatt--Furio--Lombardo \cite[Theorem 1.3]{BFL26} and Rebolledo--Wuthrich \cite[Prop. 4]{RW25}. 

\begin{lemma}
    Let $K/\Q_p$ be a finite extension with ramification index $e < p-1$ and let $E/K$ be an elliptic curve such that $\cC_{\ns}(p)(E/K)$ is not empty. Then $E$ has potentially good reduction. If $E$ has good reduction, then this reduction is supersingular and $\cC_{\ns}(p)(E/K)$ contains exactly two elements. 

\begin{proof}
    Because $\cC_{\ns}(p)(E/K) \neq \emptyset$, the image of 
    \[
    \rho_{E,p} : G_K := \mathrm{Gal}(\overline{K}/K) \ \lra \ \mathrm{Aut}(E[p](\overline{K})) \simeq \mathrm{GL}_2(\F_p)
    \]
    is contained in a non-split Cartan subgroup. Suppose $j(E) \notin \cO_K$, then by Tate uniformization the mod $p$ cyclotomic character is trivial on the inertia subgroup of $K$, which contradicts $e < p-1$.

    Assume $\cE/\cO_K$ is an elliptic curve with $\mathcal{E}_K \simeq E$. Then $\cC_{\ns}(p)(\cE/\cO_K) \neq \emptyset$ by Raynaud \cite[(3.3.3), (3.3.6)]{Ray74}, so the Hasse invariant of $\cE_{\cO_K/(p)}$ vanishes by \cref{lem:hasse-vanishes}. Serre \cite[Prop. 10(b)]{Ser72} showed that the inertia subgroup of $G_K$ has image $eC$ under $\rho_{E,p}$, where $C \leq \mathrm{GL}_2(\F_p)$ is a non-split Cartan subgroup. Since $e < p-1$, this image is non-scalar, so $\cC_{\ns}(p)(E/K)$ contains one $\Z/2\!\Z$-orbit.
\end{proof}
\end{lemma}

We now show that for $p\geq 11$, $\cX_{\ns}^+(p)$ satisfies a weak form of the N\'eron model property. 

\begin{prop}
\label{prop:moduli-scheme-is-weak-neron}

Let $p \geq 11$ be prime and $K/\Q_p$ be finite unramified. The $j$-map $\cX_{\ns}^+(p)(K) \rightarrow \cO_K$ is well-defined and injective, and the obvious map induces a bijection
\[
\cX_{\ns}^+(p)(\cO_K) \ \stackrel{\sim}{\lra} \ \cX_{\ns}^+(p)(K).
\]

\begin{proof}

We will show that every $P \in X_{\ns}^+(p)(K)$ comes from the image under the classifying map of some $(E,u)$ with $u \in \cC_{\ns}(p)(\cE_L/L)$, where $L/K$ has ramification index $e \leq 6$ and $\cE/\cO_L$ is an elliptic curve. By Raynaud's theorem \cite[(3.3.3),(3.3.6)]{Ray74}, $\cC_{\ns}(p)(\cE_L/L) \simeq \cC_{\ns}(p)(\cE/\cO_L)$, so $P$ extends to $\cX_{\ns}^+(p)(\cO_L)$. Since $\cX_{\ns}^+(p)$ is affine, $P$ extends to $\cX_{\ns}^+(p)(\cO_L \cap K) = \cX_{\ns}^+(p)(\cO_K)$ and we are done. 

Let $E/K$ be an elliptic curve with $j(E) \neq 0,1728$. Then the set of $z \in X_{\ns}^+(p)(\overline{K})$ with $j(z)=j(E)$ identifies with $\cC_{\ns}^+(p)(E/\overline{K})$ as $\mathrm{Gal}(\overline{K}/K)$-sets. Let $P,Q \in X_{\ns}^+(p)(\overline{K})$ be such that $j(P)=j(Q)=j(E)$. Let $K'/K$ be the compositum of the quadratic extensions of $K$ (which has ramification index $2$), then $P,Q$ correspond to $u,v \in \cC_{\ns}(p)(E_{K'}/K')$. In particular, $E_{K'}$ has potentially good reduction. Therefore, there exists $L_1/K$ with ramification index in $\{1,2,3,4,6\}$ such that $E_{L_1}$ has good reduction. Let $L=L_1K$, it has ramification index $e \in \{2,4,6\}$. Since $\cC_{\ns}(p)(E_{K'/K}) \subset \cC_{\ns}(p)(E_L/L)$, $u$ and $v$ are in the same $\Z/2\!\Z$-orbit so $P=Q$. 

Let $E/K$ be an elliptic curve with good reduction such that $\mathrm{Aut}(E)$ has $2w$ elements with $w \in \{2,3\}$. Then the set of $z \in X_{\ns}^+(p)(\overline{K})$ such that $j(z) = j(E)$ identifies with $\cC_{\ns}^+(p)(\overline{K})/\mathrm{Aut}(E)$ as a $G_K$-set. 

Let $P \in X_{\ns}^+(p)(K)$ with $j(P)=j(E)$. Then $P$ corresponds to $u \in \cC_{\ns}(p)(E_L/L)$ for some extension $L/K$ with degree at most $2w < p-1$. Hence $E$ has supersingular reduction and $p$ is inert in the quadratic ring $\mathrm{End}(E)$. Therefore, $\cC_{\ns}(p)(E/K)$ contains a morphism $u_0: \F_{p^2} \simeq \mathrm{End}(E)/(p) \rightarrow \mathrm{End}(E[p])$, so $u$ is in the same $\Z/2\!\Z$-orbit as $u_0$, so $P$ is the corresponding Heegner point, which clearly extends. 
\end{proof}
\end{prop}

\begin{prop}\label{prop:model}
    Let $p$ be a prime number. There exists an open immersion \[
    \cX_{\ns}^+(p) \ \lra \ X
    \]
    where $X$ is a (relatively) minimal model over $\Z$ of $X_{\ns}^+(p)_{\Q}$. When $p \geq 11$, the image of this open immersion is the smooth, non-cuspidal locus of $X$. When $p \leq 7$, there is an abstract isomorphism $X\simeq \mathbb{P}^1_{\Z}$. 

    \begin{proof}
Let $Y = \cX_{\ns}^+(p)$ for simplicity. Since $Y$ is affine, $Y$ is an open subscheme of some proper $\Z$-scheme $X$. Replacing $X$ with the scheme-theoretic image of $Y$ in $X$, we may assume that $X$ is integral, flat over $\Z$, and of dimension $2$ (cf. \cite[Proposition 13.2.10, Remarque 13.2.12]{EGAIV3}). Replacing $X$ with its normalization (possible by \cite[Scholie 7.8.3]{EGAIV2} and since $Y$ is normal), we may assume that $X$ is normal. Since $Y$ is regular, by Lipman's theorem \cite[Theorem 1.1]{Art86}, we may assume that $X$ is regular. In conclusion, there is an open immersion $Y \rightarrow X$ such that $X$ is regular connected of dimension $2$ and proper flat over $\Z$: in particular, the geometric fibres of $X \rightarrow \operatorname{Spec}(\Z)$ are connected. 

Since $X_{\Q}$ is normal projective over $\Q$, it is isomorphic to $X_{\ns}^+(p)_{\Q}$. The results of \cite[\S~10]{KM85} imply that $X_{\ns}^+(p)_{\Q}$ extends to a smooth proper curve $X^{\rm KM}$ over $\Z[1/p]$ with geometrically connected fibres. The isomorphism $X_{\Q} \simeq X_{\ns}^+(p)_{\Q}$ extends over $\Z[1/Np]$ for some $N \geq 1$ prime to $p$ by \cite[Th\'eor\`eme 8.10.5]{EGAIV3}, and we may glue $X_{\Z[1/N]}$ to $X^{\rm KM}_{\Z[1/p]}$: this scheme is regular, connected, flat, proper over $\Z$ of relative dimension $1$, hence projective \cite[Thm. 8.3.16]{QL}. In summary, we have an open immersion 
\[
Y \ \lra \ X
\]
with $X$ projective flat over $\Z$, isomorphic to $X^{\rm  KM}$ over $\Z[1/p]$, regular connected of dimension $2$. \vspace{.15cm}

\par Now assume that $X_{\F_p}$ has the smallest possible number of irreducible components. We will show that $X$ is relatively minimal. The irreducible components of $X_{\F_p}$ are of two kinds: the \emph{modular} components, which contain the image of some connected component on $Y_{\F_p}$, and the other components, which we call \emph{exotic}. If $C \subset X_{\F_p}$ is a modular component, then the modular locus of $C$ is an open subspace of $X_{\F_p}$ isomorphic as a subscheme to $\mathbb{A}^1_{\kappa}$ for some extension $\kappa/\F_p$ of degree at most two. Hence, $C$ has multiplicity one, and $C$ contains exactly one point outside the modular locus; moreover, an irreducible component $D \neq C \subset X_{\F_p}$ can only meet $C$ at that point, since the modular locus is smooth in $X_{\F_p}$. \vspace{.15cm}

Suppose that $X$ is not relatively minimal and let $E \subset X_{\F_p}$ be an exceptional divisor. Considering the contraction of $E$, it is clear by our assumption on $X$ that $E$ is a modular component with modular locus $\mathbb{A}^1_{\kappa}$, so $E \simeq \PP^1_{\kappa}$. In particular, if $x \in E$ is the point outside the modular locus, $x$ has residue field $\kappa$.  \vspace{.15cm}

\begin{itemize}
\item Assume $p \leq 7$. Then $Y_{\F_p} \simeq \mathbb{A}^1_{\F_p}$, so there is exactly one exceptional divisor $E$ on $X_{\F_p}$ which is also the only modular component, and $\kappa=\F_p$. Since $E$ is exceptional, one has 
\[
E^2 = K_{X/\Z} \cdot E = -1,
\]
where $K_{X/\Z}$ is a canonical divisor on $X$ (cf. \cite[Def. 9.1.34]{QL}). For any other irreducible component $D$ of $X_{\F_p}$, $D$ is not exceptional, so $K_{X/\Z} \cdot D \geq 0$. This contradicts the genus formula \cite[Prop. 9.1.35]{QL}. Hence $X$ is relatively minimal and it special fibre contains a component of multiplicity one; by the genus formula, $X_{\F_p}$ is irreducible. Since $Y(\Q) \neq \emptyset$, $X \rightarrow \operatorname{Spec}(\Z)$ has geometrically integral fibres and $X_{\Q} \simeq \PP^1_{\Q}$ by \cite[Prop. 7.4.1]{QL}. Therefore $X \simeq \PP^1_{\Z}$ by \cite[Ex. 8.3.5, 8.3.6]{QL}. \vspace{.15cm}

\item Assume $p \geq 11$. Then $X_{\ns}^+(p)$ has genus $g \geq 1$, and its cusps are defined over $\Q(\zeta_p)^+$. By \cref{prop:moduli-scheme-is-weak-neron} $Y_{\F_p}$ is the smooth locus of $X_{\F_p}$: indeed, any smooth point of $X_{\F_p}$ is the specialisation of a non-cuspidal point of $X(K)$ with $K/\Q_p$ finite unramified, hence in the image of $Y(\cO_K)$.

In particular, every exotic component of $X_{\F_p}$ has multiplicity at least two. By Castelnuovo's criterion, there exists a unique component $D \subset X_{\F_p}$ meeting $E$; furthermore, $D$ meets $E$ transversally at $x$, so $D$ is regular at $x$ by  \cite[Proposition 9.1.8]{QL} and $D$ has multiplicity one. In particular, $D$ is modular and its non-modular point is $x$, which is regular, so $D$ is regular, irreducible, proper over $\F_p$, and its modular locus is $\mathbb{A}^1_{\kappa'}$, so $D=\mathbb{P}^1_{\kappa'}$ and $\kappa'=\kappa=\kappa(x)$. 

Let $F \subset X_{\F_p}$ be an irreducible component distinct from $E$ or $D$. Then $F$ does not meet $(E \cup D) \backslash \{x\}$ (because it is the reunion of the modular loci of $E$ and $F$), and $x \notin F$ because otherwise $E$ would meet $F$, which is impossible since $E$ is exceptional. Therefore, $E \cup D \subset X_{\F_p}$ is closed and open, so $X_{\F_p} = E\cup D$. Since $E$ and $D$ are regular, have multiplicity one and intersect transversally, the adjunction and genus formulae imply that $g < 1$, a contradiction.

Since $g \geq 1$, a relatively minimal model is minimal, and the formation of the regular minimal model commutes with finite \'etale base change or completion \cite[Proposition 9.3.28]{QL}. \qedhere 
\end{itemize}
    \end{proof}

\end{prop}

\begin{remark}
    When $p \geq 11$, \cref{prop:model} and \cite[Theorem 1.1]{LT16} imply together that the N\'eron model of $X_{\ns}^+(p)$ over $\Z$ is obtained by glueing $\cX_{\ns}^+(p)$ and the compactification $X^{\mathrm{KM}}_{\Z[1/p]}$ of $\cX_{\ns}^+(p)_{\Z[1/p]}$. 
\end{remark}

\begin{remark}
If $\cL$ is a finite \'etale representable moduli problem over $\Z_p$, one can prove a similar minimality result for the fine moduli scheme $\cM := \cM(\cC_{\ns}^+(p),\cL)$ when $p > 3$, if for instance $\cM_{\Q_p}$ is geometrically connected and its compactification has positive genus. 

When $p \leq 3$, \cref{prop:moduli-scheme-is-weak-neron} often fails; in such cases, $\cM$ cannot be the smooth locus of any proper regular model. Indeed, if $\cL$ is a classical moduli problem, $\cM$ admits $K$-points coming from elliptic curves with split multiplicative reduction. Such points have non-integral $j$-invariant, so they do not extend to $\cO_K$.
\end{remark}

\begin{remark}
    Let $p\geq 11$. By the same proof as \cref{prop:moduli-scheme-is-weak-neron}, one can show that if $K/\Q_p$ is finite unramified and $n \geq 1$, the obvious map $\cX_{\ns}^+(p^n)(\cO_K) \rightarrow \cX_{\ns}^+(p^n)(K)$ is also a bijection. One can then follow the proof of \cref{prop:model} to show that $\cX_{\ns}^+(p^n)$ identifies with the smooth non-cuspidal locus of the minimal regular model of $X_{\ns}^+(p^n)$.  
\end{remark}

\begin{remark}
    As a consequence of this interpretation for $X_{\mathrm{EP}}^{\mathrm{sm}}(p)$, it is possible to deduce from \cite[Corollary 3.11]{EP24} a completely explicit description for the component group of the N\'eron model of the Jacobian of $X_{\ns}^+(p)$ along with its action of the Hecke algebra. Namely, this component group identifies as a module over the Hecke algebra with the quotient \[\frac{(\mathrm{Div}(\cS) \otimes \Z/(p-1)\Z)^{\deg=0}}{\Z/(p-1)\Z \cdot u},\] where $\cS$ is the set of supersingular $j$-invariants in $\overline{\F_p}$, and $u = \sum_{s \in \cS}{\frac{12}{w_s}[s]}$, where $2w_s$ is the size of the automorphism group of the elliptic curve with $j$-invariant $s$. 
\end{remark}

\section{Arithmetic intersections on mixed Cartan curves}
\label{sec:GZ} 
This section determines the arithmetic intersection number of a pair of Heegner divisors associated to coprime fundamental discriminants $\Delta_1,\Delta_2<0$ on the arithmetic surface of \cref{prop:mixed-cartan-coarse}
\[
\cX := \cX_{\Car}^+(N_{\spl},N_{\ns}).
\] 
\Cref{thm:B} was stated in the introduction for $\cX_{\ns}^+(p^n)$, but the arguments apply equally to arbitrary mixed Cartan curves and \cref{thm:B} will be proved in that generality in \cref{subsec:Heegner}. The strategy follows Gross--Zagier \cite{GZ85} and Gross--Kohnen--Zagier \cite{GKZ87}, reducing the intersection calculations to a counting problem for orders in quaternion algebras. We define mixed Cartan orders in \cref{subsec:mixed-cartan-orders}. These need not be Eichler, and we give local descriptions using the classification of Bass orders.

\textbf{Notation.} Let $B$ be a quaternion algebra over $\Q$. We write $\mathrm{disc}(B)$ for the product of ramified primes of $B$. If $S\subseteq B$ is an order, we write $\mathrm{disc}(S)$ to refer to the reduced discriminant of $S$ \cite[Definition 15.4.4]{Voi21}. The \emph{index} of $S$ will refer to the index $[R:S]$ of $S$ in any maximal order $R\subseteq B$ containing $S$; equivalently, the index of $S$ is equal to $\mathrm{disc}(S)/\mathrm{disc}(B)$.

\subsection{Residually unramified orders} We first recall the classification of residually unramified quaternion orders due to Brzezinski \cite[\S~1]{B90}. For the sake of completeness, we include arguments for statements that are not immediately extracted from Brzenzinski \cite{B90} or Voight \cite[Chapter 24]{Voi21}. 

Let $S$ be an order in $B$, and for each prime $p$ let $S_p :=S\otimes\Z_p$. Note that $\Z_{p^2} := W(\F_{p^2})$ and $\Z_p\times\Z_p$ are the only two unramified quadratic $\Z_p$-algebras up to isomorphism.

\begin{definition}
    \label{def:residually}
	Let $p\mid\mathrm{disc}(S)$. 
    \begin{itemize}
        \item We say $S$ is \emph{residually split at $p$} if $S_p$ contains a $\Z_p$-subalgebra $\cO\simeq \Z_p\times\Z_p$.
        \item We say $S$ is \emph{residually inert at $p$} if $S_p$ contains a $\Z_p$-subalgebra $\cO\simeq \Z_{p^2}$.
        \item We say $S$ is \emph{residually ramified at $p$} if every quadratic $\Z_p$-subalgebra of $S_p$ is ramified over $\Z_p$.
    \end{itemize}
	If $S$ is either residually split or residually inert at all $p\mid\mathrm{disc}(S)$ we say $S$ is \emph{residually unramified}. 
\end{definition} 

The local order $S_p$ contains an embedding of both $\Z_p\times\Z_p$ and $\Z_{p^2}$ if and only if $S_p\simeq M_2(\Z_p)$. So if $p\mid\mathrm{disc}(S)$ then exactly one of the three cases above must hold. 

\begin{remark}
    The standard definitions of residually split, inert, and ramified are expressed in terms of the isomorphism type of $S_p/J_p$, where $J_p$ is the Jacobson radical of $S_p$; see for instance \cite[Definition 24.3.2]{Voi21}. These definitions are equivalent to those given in \cref{def:residually}. In one direction, if $S_p/J_p$ is isomorphic to a quadratic $\F_p$-algebra, let $\alpha\in S_p$ be the preimage of a generator under $S_p\to S_p/J_p$. Then $\Z_p[\alpha]$ is an unramified $\Z_p$-subalgebra of $S_p$, isomorphic to either $\Z_{p^2}$ or $\Z_p\times\Z_p$ accordingly as $S_p/J_p\simeq \F_{p^2}$ or $\F_p\times \F_p$. Conversely, for any unramified quadratic $\Z_p$-subalgebra $\cO\subseteq S_p$, the map $\cO/p\cO\to S_p/J_p$ is an injection because $\cO/p\cO$ has no nonzero nilpotents.
\end{remark}

Residually unramified orders have a simple local classification, summarised in \cite[\S~24.5--6]{Voi21} for $p \neq 2$ but discussed in greater generality in \cite[\S~1]{B90}, from which the following statement may be deduced. 

\begin{lemma}
    \label{lem:resid_unram_classification}
	Let $p \mid \mathrm{disc}(S)$, and set $n:=\mathrm{ord}_p(\mathrm{disc}(S))$. Let $\cO\subseteq S_p$ be a quadratic $\Z_p$-subalgebra.
	\begin{enumerate}[label=(\alph*)]      
        \item If $\cO\simeq \Z_p\times\Z_p$ then $S_p$ is isomorphic to an Eichler order of level $p^n$ in $M_2(\Z_p)$.
        \item If $\cO\simeq\Z_{p^2}$ then there is a maximal order $R_p\subseteq B_p$ such that $S_p = \cO \ + \ p^{\lfloor n/2\rfloor}R_p$. 
    \end{enumerate}
    
\begin{proof}  
    Part (a) follows from \cite[Lemma 24.3.6]{Voi21}, so we focus on part (b). We are done if $S_p$ is maximal, so assume otherwise. Let $R_p$ denote the hereditary closure of $S_p$, so that $R_p$ is maximal and $p^n=\mathrm{disc}(S_p)=p^{2r}\mathrm{disc}(R_p)$ for some integer $r$ by \cite[24.4.9]{Voi21}. Note that $S_p$ is Bass \cite[Proposition 24.5.4]{Voi21} and non-hereditary, so we may apply \cite[Proposition 1.12(a)]{B90} to write $S_p$ as $\cO$ plus a power of the Jacobson radical of $R_p$. If $R_p\simeq M_2(\Z_p)$ (so $n=2r$) then the Jacobson radical of $R_p$ is $pR_p$, so $S_p=\cO+p^{n/2}R_p$. If instead $R_p$ is the maximal order in a division algebra (so $n=2r+1$), then its Jacobson radical is the unique maximal ideal $\mathfrak{P}$ of $R_p$, 
    and we have $S_p=\cO+\mathfrak{P}^{n-1}$. Since $\mathfrak{P}^2=pR_p$, 
    we may rewrite this as $S_p=\cO+p^{(n-1)/2}R_p$.
\end{proof}
\end{lemma}

In what follows, we will also need the classification of superorders of residually inert orders, which follows directly from either \cite[Proposition 24.4.7(b)]{Voi21} or \cite[\S~1]{B90}. 
\begin{lemma}
    \label{lem:onlysuperorders}
    Let $R_p\subseteq B_p$ be a maximal order, $\cO\subseteq R_p$ a subalgebra with $\cO\simeq\Z_{p^2}$, and $S_p=\cO+p^kR_p$. Then every superorder $S_p' \supset S_p$ in $B_p$ is of the form 
    \[S_p' = \cO+p^iR_p, \qquad \mbox{for } \ 0\leq i\leq k.\]
\end{lemma}

\subsection{Mixed Cartan orders} 
\label{subsec:mixed-cartan-orders}
We define mixed Cartan orders in quaternion algebras over $\Q$. As before, we let $N_{\spl}$ and $N_{\ns}$ be a pair of coprime positive integers, and set $N := N_{\spl}\cdot N_{\ns}$. 

\begin{definition} \label{def:mixed cartan order}
A \emph{mixed Cartan order} $S$ is a residually unramified quaternion order such that the index of $S$ is a perfect square. The \emph{level} of $S$ is defined to be the positive integer $N$ such that $[R:S]=N^2$ for any maximal order $R\supseteq S$. We say $S$ has \emph{type} $(N_{\spl},N_{\ns})$, with $N_{\spl}\cdot N_{\ns}=N$, if 
\begin{itemize}
    \item $S$ is residually split at every prime dividing $N_{\spl}$,
    \item $S$ is residually inert at every prime dividing $N_{\ns}$. 
\end{itemize}
\end{definition}

Recall from \cref{eq:quadratic_algebra} that $A(N_{\spl},N_{\ns})$ was defined as the unramified quadratic $(\Z/N\!\Z)$-algebra split at primes dividing $N_{\spl}$ and inert at primes dividing $N_{\ns}$. 

\begin{lemma}
\label{lem:maximal order of mixed cartan}
    A quaternion order $S$ is a mixed Cartan order of type $(N_{\spl},N_{\ns})$ if and only if there exists a maximal order $R$ with $NR\subset S\subset R$ and $S/NR\simeq A(N_{\spl},N_{\ns})$. Such a maximal order $R$ is then unique.
    \begin{proof}
        Let $S$ be mixed Cartan of type $(N_{\spl},N_{\ns})$. For each prime $p$ we take $R_p$ to be the hereditary closure of $S_p$, which is maximal because the index of $S_p$ in any maximal order is a square. For $p^r\parallel N_{\ns}$ we have $S_p\simeq \Z_{p^2}+p^r R_p$, so that $S_p/p^rR_p\simeq \Z_{p^2}/p^r\Z_{p^2}$, and $R_p$ is the unique maximal order containing $S_p$ by \cref{lem:onlysuperorders}. For $p^r\parallel N_{\spl}$, the Eichler order $S_p$ is associated to a path of length $2r$ on the Bruhat--Tits tree, with each maximal superorder of $S_p$ associated to a vertex on this path \cite[23.5.16]{Voi21}. For $i=0,\ldots,2r$, the inclusion of $S_p$ into the $i^\mathrm{th}$ maximal order corresponds to the inclusion of the subring
        \[
        \begin{pmatrix}
            \Z_p & p^i\Z_p \\ p^{2r-i}\Z_p & \Z_p
        \end{pmatrix} \subset M_2(\Z_p).\]
        This subring only contains $p^r M_2(\Z_p)$ if $i=r$, the vertex associated to the hereditary closure. In this case it has the form $(\Z_p\times\Z_p)+p^rM_2(\Z_p)$, and so $S_p/p^rR_p\simeq (\Z/p^r\Z)^2$. The converse is clear.
    \end{proof}
\end{lemma}

\subsection{The Clifford order $S_t$} The intersection numbers of Heegner divisors on mixed Cartan curves is encoded algebraically in the pair of embeddings of endomorphism rings
\begin{equation}
\label{eqn:embs}
\cO_1,\cO_2 \hookrightarrow R \subset B_{q\infty}    
\end{equation}
of the associated elliptic curves, where the endomorphism ring $R$ of their common reduction modulo $q$ is a mixed Cartan order. We follow Gross--Kohnen--Zagier \cite{GKZ87} and introduce the Clifford orders $S_t$, which are the smallest orders containing the images of a pair of embeddings as in \cref{eqn:embs}. The main result of this subsection is \cref{prop:quat-count}, which counts the number of mixed Cartan superorders of $S_t$. 

Let $\Delta_1,\Delta_2<0$ be two negative discriminants, and let $t \in \Z$ satisfy $t\equiv \Delta_1\Delta_2\bmod 2$ and $t^2\neq\Delta_1\Delta_2$. The binary quadratic form $\Delta_1 X^2 + 2t XY + \Delta_2 Y^2$ has nonzero discriminant $4(t^2-\Delta_1\Delta_2)$, and so its Clifford algebra over $\Q$ is a quaternion algebra $B_t$, which may be presented as
\begin{equation}
	\label{eqn:cliff}
	B_t = \Q + \Q e_1 + \Q e_2 + \Q e_1e_2 \qquad \mbox{with} \  \left\{ 
	\begin{array}{l}
		e_1^2 = \Delta_1, \, e_2^2 = \Delta_2 \\
		e_1e_2 + e_2e_1 = 2t.
	\end{array}\right.
\end{equation}
The integral elements $\gamma_i := (\Delta_i + e_i)/2$ generate quadratic orders $\cO_i$ of discriminants $\Delta_i$, where $i=1,2$. Their product $\gamma_1\gamma_2$ has trace $(\Delta_1\Delta_2+t)/2$, which is an integer by our choice of $t$. Thus $\gamma_1$ and $\gamma_2$ generate a $\Z$-order $S_t$ in the quaternion algebra $B_t$,
\begin{equation}
	\label{eqn:Sx}
	S_t := \Z + \Z\gamma_1 + \Z\gamma_2 + \Z \gamma_1\gamma_2, \qquad \mbox{with} \ \mathrm{disc}(S_t) = \frac{t^2 - \Delta_1\Delta_2}{4}. 
\end{equation}
By construction, the order $S_t$ is generated as a ring by the image of the embeddings of the quadratic orders $\cO_1$ and $\cO_2$, and every such quaternion order arises in this way for an appropriate value of $t$.

\par Now suppose $\Delta_1,\Delta_2$ are coprime fundamental discriminants. As in the introduction, for any $p$ dividing $\mathrm{disc}(S_t)$, we define $\epsilon(p)$ to be the value of the genus character $\chi$ of $\Q(\sqrt{\Delta_1\Delta_2})$ at any prime above $p$.
\begin{lemma}
    Let $\Delta_1,\Delta_2<0$ be a pair of coprime fundamental discriminants. Then the Clifford order $S_t$ is residually unramified. More precisely, at any prime $p \mid \mathrm{disc}(S_t)$ we have 
    \begin{itemize}
        \item $S_t$ is residually split if and only if $\epsilon(p)=+1$,
        \item $S_t$ is residually inert if and only if $\epsilon(p)=-1$.
    \end{itemize}
\begin{proof}
    Since $\Delta_1,\Delta_2$ are coprime, we can pick $i=1$ or $2$ so that $\cO := \cO_i$ is unramified at $p$. Then $\cO\otimes \Z_p$ is isomorphic to $\Z_p\times \Z_p$ if $\epsilon(p)=1$, and is isomorphic to $\Z_{p^2}$ if instead $\epsilon(p)=-1$. 
\end{proof}
\end{lemma}

\par For the proof of \cref{thm:B} in the next section, we need to count the number of mixed Cartan orders of level $N$ containing $S_t$. Note that $\mathrm{disc}(S_t)$ must be divisible by $N^2$ for this set to be non-empty. 

\begin{definition}
    Given $t\in\Z$ with $t\equiv \Delta_1\Delta_2\bmod 2$ and $t^2\neq \Delta_1\Delta_2$, we define $\mathrm{Car}_t(N_{\spl},N_{\ns})$ to be the set of mixed Cartan orders $R$ of type $(N_{\spl},N_{\ns})$ (cf.~\cref{def:mixed cartan order}) with $S_t\subset R \subset S_t\otimes\Q$.
\end{definition}

\begin{prop}
\label{prop:quat-count}
Assume $\Delta_1,\Delta_2 < 0$ are two Heegner discriminants for $\mathrm{Car}(N_{\spl},N_{\ns})$. Let $S_t$ be the quaternion order defined by \cref{eqn:Sx}, and assume that $N^2 \mid \mathrm{disc}(S_t)$. Then 
\[
|\mathrm{Car}_t(N_{\spl},N_{\ns}) | = \prod_{\substack{p \,\mid \, \mathrm{disc}(S_t)\\ \epsilon(p) = 1}} \left( 1 + \mathrm{ord}_p\left(\frac{\Delta_1\Delta_2 - t^2}{4N^2}\right)\!\right) 
\]
\begin{proof}
It suffices to count the number of mixed Cartan superorders locally. When $p \nmid \mathrm{disc}(S_t)$ the order $S_t$ is maximal at $p$, so the number is one. When $p\mid \mathrm{disc}(S_t)$ the local structure of 
\[
S_{t,p}:=S_t\otimes\Z_p
\]
is determined in \cref{lem:resid_unram_classification}. We distinguish three cases. \vspace{.1cm}
\begin{itemize}
\item Suppose $p \mid N_{\spl}$. Then we must have $\epsilon(p) = +1$, so $S_{t,p}$ is an Eichler order of level $p^n$ where $n=\mathrm{ord}_p(\mathrm{disc}(S_t))$. If $p^k\parallel N_{\spl}$, then locally at $p$, a mixed Cartan order is Eichler of level $p^{2k}$. Since Eichler orders of level $p^n$ correspond to paths of length $n+1$ in the Bruhat--Tits tree \cite[23.5.16]{Voi21}, the number of Eichler orders of level $p^{2k}$ containing an Eichler order of level $p^n$ is
\[
n+1-2k=1+\mathrm{ord}_p(\mathrm{disc}(S_t)/N^2).
\]

\item Suppose $p \mid N_{\ns}$. Then we must have $\epsilon(p)=-1$, so writing $\mathrm{disc}(S_{t,p})=p^n$, we can conclude that $S_{t,p}=\cO+p^{\lfloor n/2\rfloor}R_p$ for some maximal order $R_p$ and quadratic subalgebra $\cO\simeq\Z_{p^2}$. If $p^k\parallel N_{\ns}$, then $\cO+p^kR_p$ is a mixed Cartan order of the correct type at $p$, and it is in fact the unique superorder of $S_{t,p}$ with the correct index, by \cref{lem:onlysuperorders}.\vspace{.1cm}
\item Suppose $p \nmid N$. Then a mixed Cartan order of level $N$ is maximal at $p$. If $\epsilon(p)=+1$ then as before the local order $S_{t,p}$ is Eichler and its number of maximal superorders equals
\[
1 + \mathrm{ord}_p(\mathrm{disc}(S_t))= 1 + \mathrm{ord}_p(\mathrm{disc}(S_t)/N^2).
\]
If $\epsilon(p)=-1$ then there is a unique maximal order containing $S_{t,p}$ by \cref{lem:onlysuperorders}.
\end{itemize}
The statement of the proposition follows by taking the product of all these local counts. 
\end{proof}
\end{prop}

\subsection{Intersection numbers of Heegner divisors}\label{subsec:Heegner} 
Let $\cX:=\cX_{\Car}^+(N_{\spl},N_{\ns})$ as in \cref{prop:mixed-cartan-coarse}.
Consider a 
Heegner discriminant $\Delta < 0$ for $\cX$ as in \cref{eqn:Heegner-disc}, and let $\cO$ be the quadratic order of discriminant $\Delta$. For any invertible ideal $\fa$ of $\cO$ we have a complex elliptic curve $E_{\fa} := \C/\fa$ with endomorphism ring $\cO$, with a canonical injection 
\[
\cO\! / N\!\cO \ \hookrightarrow \ \mathrm{End}_{\C}(E_{\fa}[N]). 
\]
The ring $\cO\! / N\!\cO$ may be identified with the quadratic \'{e}tale algebra $A := A(N_{\spl},N_{\ns})$ of \cref{eq:quadratic_algebra}. 
This identification is canonical up to automorphisms, giving a well-defined point $Q_{\fa}$ on $\cX(\C)$ that depends only on the class of $\fa$ in $\mathrm{Pic}(\cO)$; we call such $Q_{\fa}$ a \emph{Heegner point}. By the theory of complex multiplication, $Q_{\fa}$ is defined over a number field of degree $h(\cO)$ inside the ring class field of $\cO$, and the Heegner divisor 
\[
P_{\Delta} := \sum_{[\fa] \in \mathrm{Pic}(\cO)} Q_{\fa} \quad \in \mathrm{Div} \, \cX(\Q)
\] 
is defined over $\Q$. There exists a number field $L$ over which $E_{\fa}$ has all its endomorphisms defined, and has good reduction at all finite places. We obtain a point $Q_{\fa} \in \cX(\cO_{L})$, and $P_{\Delta}$ is defined as a Cartier divisor, finite flat over $\Z$, so that the arithmetic intersection of two Heegner divisors on $\cX$ is defined.

\begin{remark}
    A technical complication is that the above description of the Heegner divisor uses the moduli problem $\cP_A^+$, which is only defined over $\Z[1/N_{\ns}]$. To resolve this, note that $Q_{\fa}\in \cX(L)$ comes from a $\cO_L[1/N_{\spl}]$-point $C_{\fa}$ of $Y := \cC^+(N_{\spl},N_{\ns})_{E_{\cO_L}/\cO_L}$ using the isomorphism of \cref{cor:PA moduli problem}. Since $Y$ is finite above $\cO_L[1/N_{\ns}]$ by construction, the valuative criterion of properness ensures that $C_{\fa}$ extends to an $\cO_L$-point of $Y$, and therefore $Q_{\fa}$ extends to a point in $\cX(\cO_L)$. Let $\fq \subset \cO_L$ be a prime dividing $N_{\spl}$, so $Q_{\fa}$ sends $\fq$ to a closed point $x \in \cX$. There is not much one can say about the level structure determined by $x$, but all we will need is the fact that the underlying elliptic curve at $x$ is the reduction of $E_{\fa}$ mod $\fq$.     
\end{remark}

\par Now consider a pair of Heegner points $Q_1,Q_2$ with 
Heegner discriminants $\Delta_1,\Delta_2<0$, and let $\mathbf{E}_1,\mathbf{E}_2$ be a pair of associated elliptic curves $E_1,E_2$ enriched with the Cartan-plus level structure 
induced by the obvious maps $\cO_i \to\End(E_i[N])$. 
Let $q$ be a prime, and let $\mathfrak{q}$ be a prime above $q$ in the compositum $H$ of the two ring class fields of discriminants $\Delta_1,\Delta_2$. By \cite[Lemma 2.4]{Con04} the points $Q_1,Q_2$ are defined over $W$, the integers of the completion of the maximal unramified extension of $H_{\mathfrak{q}}$, and we may assume that $E_1,E_2$ have good reduction over $W$ with all their geometric endomorphisms defined over $W$. 
For the remainder of this section $\overline{\F}_q$ refers to the residue field of $W$ (which is indeed an algebraic closure of $\F_q$).

\begin{lemma}\label{lem:int equals isoms}
Suppose $q\nmid N_{\spl}$, and that $Q_1,Q_2$ have the same reduction $\overline{Q}_1 = \overline{Q}_2$ in $\cX(\overline{\F}_q)$. Let $x\in \mathrm{Max}\cX_W$ be the associated closed point. Then
    \begin{align}
    \mathrm{length}\left(\frac{\cO_{\cX_W,x}}{(Q_1,Q_2)} \right) =  \frac{1}{2} \sum_{n \geq 1} |\mathrm{Isom}_{W/\mathfrak{q}^nW}(\bE_1,\bE_2)|. 
\end{align}

\begin{proof}
Since $q\nmid N_{\spl}$, $\cX_W$ is a smooth quotient of a smooth fine moduli space by \cref{prop:mixed-cartan-coarse}. So this follows from the arguments of Gross--Zagier \cite[Prop. 6.1]{GZ86} or Conrad \cite[Thm. 4.1]{Con04}. 
\end{proof}
\end{lemma}

    For the remainder of \cref{subsec:Heegner} we assume that $\Delta_1,\Delta_2$ are coprime and fundamental. If the Heegner points $Q_1,Q_2$ have the same reduction in $\cX(\overline{\F}_q)$, then using the $j$-invariant map $\cX\to \mathbb{P}^1$, we can conclude $(E_1)_{\overline{\F}_q}\simeq (E_2)_{\overline{\F}_q}$.    
    Then $\mathrm{End}_{\overline{\F}_q}(E_1)$ contains copies of $\cO_1$ and $\cO_2$, so $(E_i)_{\overline{\F}_q}$ is supersingular. Hence $q$ splits in neither $\cO_i$, and in particular $q \nmid N_{\spl}$. Now fix generators $u_i = (\Delta_i+\sqrt{\Delta_i})/2$ of $\cO_i$. Let
    \[
    j: (\mathbf{E}_1)_{\overline{\F}_q} \ \stackrel{\sim}{\lra} \ (\mathbf{E}_2)_{\overline{\F}_q}
    \]
    be an isomorphism of elliptic curves enriched with $\cP_A^+$-structure. By construction, we have a ring homomorphism $\cO_1 \otimes \cO_2 \rightarrow W$. Using the morphisms $\mathrm{End}(E_i) \rightarrow W$ given by the Lie action, we may canonically identify $\cO_i$ and $\mathrm{End}(E_i)$. Let $u'_2 = j^{-1}u_2j \in \mathrm{End}(E_1)$, so that 
    \[u'_2 \in R_1 := \cO_1+N\mathrm{End}_{\overline{\F}_q}(E_1).\]  Define $t \in \Z$ by $\mathrm{Tr}(u_1u'_2) =(\Delta_1\Delta_2+t)/2$. The rule $(\gamma_1,\gamma_2) \mapsto (u_1,u'_2)$ defines a morphism $\phi$ from the Clifford order $S_t$ of \cref{eqn:Sx} to $R_1$. In particular, $(t^2-\Delta_1\Delta_2)/4$ is negative and divisible by $\mathrm{disc}(R_1)=N^2q$. Now $R := (\phi \otimes \Q)^{-1}(R_1)$ is a mixed Cartan order in $S_t \otimes \Q$ of type $(N_{\spl},N_{\ns})$, to which the maximal order attached by \cref{lem:maximal order of mixed cartan} is 
    \[
    \OO(R) := (\phi \otimes \Q)^{-1}(\End_{\overline{\F}_q}(E_1)).
    \]
    Finally, $\OO(R)$ is oriented by the map $\omega: z \in \OO(R) \mapsto \mathrm{Lie}(\phi(z)) \in \overline{\F}_q$. Thus we have defined a map 
\begin{equation}
\Psi_W\colon \!\! \left\{\left(E_1,\, E_2,\, j\colon \!(\mathbf{E}_1)_{\overline{\F}_q} \!\overset{\sim}{\rightarrow} (\mathbf{E}_2)_{\overline{\F}_q}\right) \right\} \longrightarrow \left\{(t,R,\omega) : \small \!\!\begin{array}{ll} S_t \otimes \Q \simeq B_{q\infty}, & \!\!\! R \in \mathrm{Car}_t(N_{\spl},N_{\ns}),\\ \omega: \OO(R) \to \overline{\F}_q, & \!\!\! \omega(\gamma_i) = u_i \bmod{\mathfrak{q}}\end{array}\!\! \right\},
\end{equation}
where the $E_i$ run over the isomorphism classes of elliptic curves over $W$ with endomorphism ring $\cO_i$. 

\begin{lemma}
\label{lem:psiW-fibres}
The map $\Psi_{W}$ is surjective and its fibres have cardinality $\frac{1}{2}\cdot |\cO_1^{\times}| \cdot |\cO_2^{\times}|$. 

\begin{proof}
    Let $(t,R,\omega)$ be in the range of $\Psi_W$ with $\omega: \OO \rightarrow \overline{\F}_q$. The oriented order $(\OO,\omega)$ corresponds to a supersingular elliptic curve $E_0/\overline{\F}_q$ (cf. \cite[Proposition 42.4.8]{Voi21}). The two couples $(E_i,u_i)$ lifting $(E_0,\phi(\gamma_i))$ are then constructed by \cite[Prop. 2.7]{GZ85}. 
    
    If $(E_1,E_2,j)$ is in the domain of $\Psi_W$ and $\alpha_i \in \Aut(E_i)$, then $\Psi_W(E_1,E_2,\alpha_2j\alpha_1) = \Psi_W(E_1,E_2,j)$. Therefore, we are reduced to showing that if $(t,R,\omega) = \Psi_W(E_1,E_2,j)=\Psi_W(E'_1,E'_2,j')$, there are isomorphisms $f_i: E_i \rightarrow E'_i$ such that $j' f_1 \in \mathrm{Aut}(E'_2)f_2 j \mathrm{Aut}(E_1)$. Indeed, we have an isomorphism $(E_1)_{\overline{\F}_q} \simeq (E'_1)_{\overline{\F}_q}$ because they have the same oriented endomorphism rings (cf. \cite[Prop. 42.4.8]{Voi21}), and the conclusion follows by using carefully the uniqueness in \cite[Prop 2.7]{GZ85}. 
\end{proof}
\end{lemma}

\begin{lemma}
\label{lem:int-mult-psiW}
Let $n\geq 1$. If $(t,R,\omega) = \Psi_W(E_1,E_2,j)$, then $j$ lifts to an isomorphism $(\mathbf{E}_1)_{W/\mathfrak{q}^n} \rightarrow (\mathbf{E}_2)_{W/\mathfrak{q}^n}$ if and only if $q^{2n-1} \mid \frac{\Delta_1\Delta_2-t^2}{4N^2}$. Such a lift is always unique. 

\begin{proof}
    The uniqueness is well-known (e.g. \cite[Th. 2.4.2]{KM85}). If $q \mid \Delta_1\Delta_2$, then $(E_1)_{W/\mathfrak{q}^2} \not\simeq (E_2)_{W/\mathfrak{q}^2}$ because their $q$-divisible groups are not isomorphic by \cite[Proposition 3.3]{Gro86}, so both conditions are equivalent to $n \leq 1$. Hence we may assume that $q \nmid \Delta_1\Delta_2$, so $\mathfrak{q}W=qW$.  
    
    Assume that $j$ lifts to $j': (\mathbf{E}_1)_{W/\mathfrak{q}^n} \rightarrow (\mathbf{E}_2)_{W/\mathfrak{q}^n}$. Then by Serre--Tate \cite[Th. 2.9.1]{KM85} and \cite[Prop. 3.3]{Gro86} one has 
    \[j^{-1}u_2j \in \cO_1+N\End_{W/\mathfrak{q}^nW}(E_1) = \cO_1+Nq^{n-1}\End_{\overline{\F}_q}(E_1),\] so \[N^2q^{2n-1} = \left|\mathrm{disc}\left[\cO_1+Nq^{n-1}\End_{\overline{\F}_q}(E_1)\right]\right| \mid |\mathrm{disc}(S_t)|=\frac{\Delta_1\Delta_2-t^2}{4}.\]

    Conversely, assume that $N^2q^{2n-1} \mid (\Delta_1\Delta_2-t^2)/4$. Let $R_1 = \cO_1+N\End_{\overline{\F}_q}(E_1)$. By considering the $\cO_1 \otimes \Z_q$-module structure, one sees that 
    \[\phi(S_t) \otimes \Z_q = (\cO_1 \otimes \Z_q)+q^t(R_1 \otimes \Z_q)\] with $t \geq n-1$. Therefore, 
    \[j^{-1}u_2j \in \cO_1+q^{n-1}N\End_{\overline{\F}_q}(E_1),\] so by Serre--Tate and \cite[Prop. 3.3]{Gro86} $j^{-1}u_2j$ lifts to some $u'_2 \in \End_{W/\mathfrak{q}^nW}(E_1)$. By \cite[Prop. 2.7]{GZ85}, $((E_1)_{W/\mathfrak{q}^n},u'_2)$ lifts to some $(E'_2/W,v_2)$ where $\mathrm{Lie}(v_2)=u_2$. Since $(E'_2/W,v_2)$ and $(E_2/W,u_2)$ are lifts of $((E_1)_{\overline{\F}_q},j^{-1}u_2j)$, they are isomorphic, so $(E_2/W,j^{-1}u_2j)$ lifts $((E_1)_{W/\mathfrak{q}^n},u'_2)$, which implies that $j$ lifts to $W/\mathfrak{q}^n$. 
\end{proof}
\end{lemma}

As a corollary of the above results, we obtain the main theorem of this section, which is a generalisation of \cref{thm:B} stated in the introduction to arbitrary mixed Cartan curves. 

\begin{thmx}\label{thm:C}
Let $\Delta_1,\Delta_2 < 0$ be a pair of coprime fundamental Heegner discriminants for the mixed Cartan curve $\cX$. Let $P_1$ and $P_2$ be their associated Heegner divisors, and $w_i := |\cO_i^{\times}|$. Then
\vspace{4pt}\begin{align}
\label{eqn:ThmCGZ}
\langle P_1, P_2 \rangle \ \ &= 
 \ \ \frac{w_1w_2}{8} \ \cdot \!\!\!\!\!\!\!\sum_{\substack{t \in \Z,\;d,d'>0 \\ t^2+4N^2dd'=\Delta_1\Delta_2}} \epsilon(d') \log(d).
\end{align}
\begin{remark}
    The Heegner condition forces $N$ to be coprime to $\Delta_1\Delta_2$, so if $N>1$ then there is no contribution from $t=0$. In this case the terms come in pairs $t,-t$ which produce the same value $m:=\frac{\Delta_1\Delta_2-t^2}{4N^2}$, from which \cref{thm:B} follows. For $N=1$, $t=0$ contributes if and only if $\{\Delta_1,\Delta_2\}=\{-4,-q\}$ for an odd prime $q$, in which case the summand is $\log(q)$; see e.g.~the first equation of \cite{GZ85}. 
\end{remark}
\begin{proof}
We first consider the left-hand side of \cref{eqn:ThmCGZ}, which by \cref{lem:int equals isoms} is
\begin{equation}\label{eqn:CGZ lhs 1}
\sum_{q \text{ prime}}{\frac{\log(q)}{2e_{W/\Q_q}}\sum_{\,\,E_1,E_2/W \phantom{\overset{\sim}{\rightarrow}}\!\!\!}{\sum_{j: (\mathbf{E}_1)_{\overline{\F}_q} \overset{\sim}{\rightarrow} (\mathbf{E}_2)_{\overline{\F}_q}}{m(j)}}},
\end{equation}
where $m(j)$ denotes the highest $n \geq 1$ such that $j$ lifts to an isomorphism of enriched elliptic curves over $W/\fq^n$. 
Here, the notations $W,\mathfrak{q}$ have the same meanings as above for any given prime $q$. By \cref{lem:psiW-fibres,lem:int-mult-psiW}, the sum \cref{eqn:CGZ lhs 1} can be rewritten as 
\begin{equation}\label{eqn:contrib-q}
\frac{w_1w_2}{8}\sum_{q\text{ prime}}\log(q)\!\!\!\sum_{\substack{t \in \Z\\ S_t \otimes \Q \, \simeq \, B_{q\infty}}}\!\!\!{\frac{1}{e_{W/\Q_q}}\left(1+\mathrm{ord}_q\left(\frac{\Delta_1\Delta_2-t^2}{4N^2}\right)\right)|\{(R,\omega) : (t,R,\omega) \in \mathrm{im}(\Psi_W)\}|}.
\end{equation}

We next consider the right-hand side of \cref{eqn:ThmCGZ}. For a fixed integer $t$, let $m:=\frac{\Delta_1\Delta_2-t^2}{4N^2}$; note that $t$ only contributes to the sum if $m$ is a positive integer.
As in \cite[p.192]{GZ85}, one has 
\[
\sum_{d,d'>0,\; dd'=m} \epsilon(d')\log(d)\;=\; \frac{1}{2}\left(\mathrm{ord}_q\left(m\right)+1 \right) \log(q)\ \!\cdot\!
\displaystyle\prod_{\substack{\epsilon(p) = 1\\p \mid m}} \left(\mathrm{ord}_p\left(m\right) + 1\right)\] if the set
\[\mathrm{diff}(m) := \{ p \mid m \mbox{ prime}: \epsilon(p) = -1,\; \mathrm{ord}_p(m) = \, \mbox{odd}\}\] is a singleton $\{q\}$, and the sum vanishes otherwise. 
Now for each prime $q$, we have $\mathrm{diff}(m) = \{q\}$ if and only if $S_t \otimes \Q$ is isomorphic to the quaternion algebra $B_{q\infty}$. 
Therefore, by \cref{prop:quat-count}, the right-hand side of \cref{eqn:ThmCGZ} can be rewritten as 
\begin{equation}\label{eqn:CGZ rhs 1}
\frac{w_1w_2}{8}\sum_{q \text{ prime}} \log(q) \!\!\!\sum_{\substack{t \in \Z,\\S_t \otimes \Q \simeq B_{q\infty}}} \!\!\frac{1}{2}\left(1+\mathrm{ord}_q\left(\frac{\Delta_1\Delta_2-t^2}{4N^2}\right)\right)|\mathrm{Car}_t(N_{\spl},N_{\ns})|,
\end{equation}
because $\mathrm{Car}_t(N_{\spl},N_{\ns})$ is empty if $N^2 \nmid \mathrm{disc}(S_t)$. 
Comparing \cref{eqn:contrib-q} to \cref{eqn:CGZ rhs 1}, we see that the theorem holds once we can show, for each prime $q$ and $t\in\Z$ with $S_{t}\otimes \Q\simeq B_{q\infty}$, that
\[\sum_{r\in\{t,-t\}}\frac{1}{e_{W/\Q_q}}|\{(R,\omega) : (r,R,\omega) \in \mathrm{im}(\Psi_W)\}|=\sum_{r\in\{t,-t\}}\frac12|\mathrm{Car}_r(N_{\spl},N_{\ns})|.\]

Assume that $q \mid \Delta_1$, so $q \nmid \Delta_2$ and $e_{W/\Q_q}=2$. Let $R \in \mathrm{Car}_t(N_{\spl},N_{\ns})$. Let $\OO$ be the maximal order associated to $R$ by \cref{lem:maximal order of mixed cartan}. The two ring homomorphisms $\omega: \OO \rightarrow W/\mathfrak{q}$ are conjugate under Frobenius. Both satisfy $\omega(\gamma_1)=u_1 \mod{\mathfrak{q}}$, and exactly one of them satisfies $\omega(\gamma_2) = u_2 \mod{\mathfrak{q}}$, and the conclusion follows. The same holds if $q \mid \Delta_2$.

Now assume $q \nmid \Delta_1\Delta_2$ and let $t \in \Z$ be such that $S_t \otimes \Q \simeq B_{q\infty}$. Let $J_t$ be the set of couples $(R,\omega)$ with $R \in \mathrm{Car}_t(N_{\spl},N_{\ns})$, and $\omega$ is a $W/\fq$-valued orientation of the maximal order attached to $R$ by \cref{lem:maximal order of mixed cartan} sending $\gamma_1$ to $u_1\mod{\fq}$. One has clearly $|J_t| = |\mathrm{Car}_t(N_{\spl},N_{\ns})|$. 

Let $I: S_t \rightarrow S_{-t}$ be the isomorphism respecting the first canonical generator $\gamma_1$, but sending the second canonical generator $\gamma_2$ of $S_t$ to the quadratic conjugate of the second canonical generator of $S_{-t}$. Then $I$ clearly induces a bijection $I: J_t \rightarrow J_{-t}$; furthermore, if $(R,\omega) \in J_t$ and its image under $I$ is $(R',\omega')$, then $(t,R,\omega) \in \mathrm{im}(\Psi_W)$ if and only if $(-t,R',\omega') \notin \mathrm{im}(\Psi_W)$. Thus one has 
\begin{align*}
    |\{(R,\omega) : (-t,R,\omega) \in \mathrm{im}(\Psi_W)\}| &= |\{(R,\omega) \in J_t: (t,R,\omega) \notin \mathrm{im}(\Psi_W)\}|,\text{ hence} \\
    \sum_{r \in \{\pm t\}} |\{(R,\omega): (r,R,\omega) \in \mathrm{im}(\Psi_W)\}| &= |J_t|= \sum_{r \in \{\pm t\}}{\frac{1}{2}|\mathrm{Car}_r(N_{\spl},N_{\ns})|},
\end{align*}
and we are done. \qedhere
\end{proof} 
\end{thmx}

\subsection{Non-fundamental, non-coprime discriminants}
\label{rmk:non-fundamental}
As was the case for counting maximal orders in Gross--Zagier \cite{GZ85} and Eichler orders in Gross--Kohnen--Zagier \cite{GKZ87}, the assumption that $\Delta_1$ and $\Delta_2$ are fundamental and coprime is crucial. For maximal orders, Lauter--Viray \cite{LV15} obtain formulae without this assumption. It is not our ambition to extend this in full to mixed Cartan orders, but we make note of a few simple generalizations here.

\begin{prop}\label{prop:non-fundamental}
    Let $\Delta_1,\Delta_2<0$ be coprime discriminants. Write $\Delta_1=f_1^2D_1$ and $\Delta_2=f_2^2D_2$ for fundamental discriminants $D_1$ and $D_2$. If every element of $\mathcal{S}(\Delta_1\Delta_2,N^2)$ is coprime to $f_1f_2$, then \cref{eqn:ThmCGZ} holds.
\begin{proof}
    The proof of \cref{thm:C} goes through with the following modifications. By the discussion following \cref{lem:int equals isoms}, if there is an arithmetic intersection at a prime $q$, then $q \mid m$ for some $m \in \cS(\Delta_1\Delta_2,N^2)$. Hence $q \nmid f_1f_2$, so that $e_{W/\Q_q} \leq 2$ (with equality iff $q \mid \Delta_1\Delta_2$). Hence \cite[Proposition 3.3]{Gro86} still applies. Moreover, in the proof that $\Psi_W$ is onto in \cref{lem:int-mult-psiW}, we need to check that the quasi-canonical lifts have CM by the order $\cO_i$ rather than a proper super-order. This is true because, for any super-order $R \supset S_t$, the embedding $\cO_i \hookrightarrow S_t \rightarrow R$ is optimal: indeed, it is optimal at every prime $r \nmid f_i$ because $(\cO_i)_r$ is normal, and optimal at every prime $r \nmid |\mathrm{disc}(S_t)| \in \cS(\Delta_1\Delta_2,N^2)$ because $\cO_i \hookrightarrow S_t$ is optimal. 
\end{proof}
\end{prop}

If $\Delta_1$ and $\Delta_2$ are not coprime then the situation is much more subtle, but for the classification of arithmetic intersections of CM points in $X_{\rm ns}^+(p)(\Q)$ in \cref{subsec:RW}, we can use the following weaker result. From now on we assume $p\geq 11$. By Lemma 6 and Proposition 10 of \cite{RW25} (for $\Delta\notin\{-3,-4\}$ and $\Delta\in \{-3,-4\}$, respectively), this implies that every point in $X_{\ns}^+(p)(\Q)$ with CM $j$-invariant is a Heegner point.

\begin{lemma}
\label{lem:no_intersect}
Let $\Delta_1,\Delta_2 < 0$ be discriminants and $q$ a prime. Suppose $\Delta_1/\Delta_2\neq q^{2r}$ for any $r \in \Z$. If their Heegner divisors on $\cX_{\ns}^+(p)$ intersect at $q$, then $\mathcal{S}(\Delta_1\Delta_2,1)$ contains an element divisible by $p^2q$.
\begin{proof}
If $P_1$ and $P_2$ intersect at $q$, then as in the discussion in \cref{subsec:Heegner}, this implies that there exists an elliptic curve $E/\overline{\F}_q$ and injections $\cO_1,\cO_2\hookrightarrow R\subseteq \End_{\overline{\F}_q}(E)$ for some mixed Cartan order $R$ of level $p$. If $\Delta_1/\Delta_2$ is not a square then $\cO_1$ and $\cO_2$ have non-isomorphic fraction fields. If $\Delta_1/\Delta_2$ is a square, then since we are assuming $\Delta_1/\Delta_2$ is not an even power of $q$, there exists a prime $\ell\neq q$ such that $\Delta_1$ and $\Delta_2$ have distinct $\ell$-adic valuations. The embeddings $\cO_1,\cO_2\hookrightarrow \End_{\overline{\F}_q}(E)$ are both optimal at $\ell$ \cite[Proposition 2.2]{LV15}, so they must land in distinct subfields. 

In either case we may conclude that $\cO_1$ and $\cO_2$ land in distinct subfields of $B_{q\infty}$, so they generate a quaternion order of the form $S_t$ for some value of $t$. Since $S_t$ is definite, we have $t^2 < \Delta_1\Delta_2$, and since $S_t$ is contained in a mixed Cartan order of level $p$ in $B_{q\infty}$, we must have $qp^2 \mid \mathrm{disc}(S_t)$, so that $\mathrm{disc}(S_t)$ defines an element of $\mathcal{S}(\Delta_1\Delta_2,1)$ divisible by $p^2q$. 
\end{proof}
\end{lemma} 

If $\Delta_1/\Delta_2=q^{2r}$ for some integer $r$, then $P_1$ and $P_2$ can intersect at $q$ even if $\mathcal{S}(\Delta_1\Delta_2,1)$ has no elements divisible by $p^2q$. Let $\Delta$ be the ``$q$-fundamental'' discriminant dividing both $\Delta_1$ and $\Delta_2$ (that is, $\Delta_1$ divided by the maximal even power of $q$ such that the result is still a discriminant). Then by \cite[Proposition 2.2]{LV15}, both embeddings $\cO_1,\cO_2\to\mathrm{End}(E)$ land in an order of discriminant $\Delta$, so it is possible that they land in the same subfield of $B_{q\infty}$. We investigate such \emph{degenerate intersections} in the following lemma.

\begin{lemma}\label{lem:degenerate}
    For each $i=1,2$, let $E_i/W$ be an elliptic curve with CM by an order $\cO_i$ of discriminant $\Delta_i$ in which $p$ is inert. Let each $E_i$ be equipped with the $\mathcal{C}_{\ns}^+(p)$ level structure arising from the map $\cO_i\to \End(E_i[p])$.
    Suppose that $\Delta_1=q^{2r}\Delta_2$ for some $r>0$, and that $\mathcal{S}(\Delta_1\Delta_2,p^2)$ is empty. Then an isomorphism $E_1\to E_2$ over $W/\fq^n$ preserves the level structures 
    if and only if the induced isomorphism 
    \[
    \phi : \End_{W/\fq^n}(E_1)\ \stackrel{\sim}{\lra} \ \End_{W/\fq^n}(E_2)
    \]
    has the property that $\phi(\cO_1) \subset \cO_2$. 
    \begin{proof}
        First suppose an isomorphism $E_1\to E_2$ over $W/\fq^n$ preserves the $\mathcal{C}_{\ns}^+(p)$ level structure. Since $\mathcal{S}(\Delta_1\Delta_2,p^2)$ is empty, there cannot exist any non-split Cartan order containing linearly independent copies of $\cO_1$ and $\cO_2$, so the isomorphism $\End_{W/\fq^n}(E_1)\to \End_{W/\fq^n}(E_2)$ must send $\cO_1$ into $\cO_2$.
        
        Conversely, suppose we have an isomorphism over $W/\fq^n$ sending $\cO_1$ into $\cO_2$. The images of $\cO_1$ and $\cO_2$ in $R:=\End_{W/\fq^nW}(E_2)$ generate the same subalgebra of $R/pR$ because both conductors are coprime to $p$. This subalgebra determines the non-split Cartan structure on both $E_1$ and $E_2$, so the isomorphism preserves the Cartan structure.
    \end{proof}
\end{lemma}

\section{Intersections of rational CM points}
\label{subsec:RW}
Rebolledo--Wuthrich \cite{RW25} compute with rational points on $X_{\ns}^+(p)$ using a moduli interpretation over $\Z[1/6p]$ in terms of \textit{necklaces}, see \cite{RW18}. As an application, they determine for $5 \leq p \leq 47$, and for any pair of Heegner discriminants of class number one  
\begin{equation}
\label{eqn:classno1}
\Delta_1,\Delta_2 \ \in \ \{ -3, -4, -7, -8, -11, -12,-16, -19, -27,-28, -43, -67, -163 \}
\end{equation}
a list of primes $q \neq 2,3,p$ for which the corresponding CM points on $X_{\ns}^+(p)$ have the same reduction modulo $q$. They expect that this is the full list of all such intersections, for any $p$. We confirm their results by computing the intersection multiplicities at all primes $q$ (including $q=2,3,p$) of the CM points in $X_{\ns}^+(p)(\Q)$ for all primes $p$. The ``intersection numbers'' $\exp(\langle P_1,P_2\rangle)$ are recorded in \cref{fig:Xnsp_intersections}.

\subsection{Genus zero}

For each $p=2,3,5,7$, we provide an integral model for $M\simeq \mathbb{P}^1_{\Z}$ for $X_{\ns}^+(p)$ in \cite{github}. Specificially, we provide a morphism $j:\mathbb{P}^1_{\Z}\to \mathbb{P}^1_{\Z}$ whose restriction to the generic fibre factors as an isomorphism $\phi:\mathbb{P}^1_{\Q}\to X_{\ns}^+(p)$ followed by the $j$-invariant map $X_{\ns}^+(p)\to\mathbb{P}^1_{\Q}$ (taken from \cite{LMFDB}), and we use $\phi$ to identify the domain of $j$ as an integral model for $X_{\ns}^+(p)$.
The intersection numbers in \cref{fig:Xnsp_intersections} were computed using this model. For each pair $\Delta_1,\Delta_2$ of coprime fundamental Heegner discriminants, these numbers --- the blue entries in the table --- were confirmed to agree with the predictions of \cref{thm:B} (with one caveat, see \cref{rmk:nonheegner}). 

Now let $\cX_{\ns}^+(p) \hookrightarrow X := \PP^1_{\Z}$ be an open immersion as in \cref{prop:model}. To show that the intersection numbers in \cref{fig:Xnsp_intersections} agree with those on $\cX_{\ns}^+(p)$, it suffices to confirm, for each prime $q$, that $M_{\Z_q}\simeq X_{\Z_q}$ as models of $X_{\ns}^+(p)$. We prove this by showing that both models share a certain distinguishing property. 
\begin{lemma}\label{lem:unique_P1}
    Let $x,y,z \in \PP^1(\Q_q)$ be distinct points, and $m \geq 0$. Then $\PP^1_{\Q_q}$ has a model $M \simeq \PP^1_{\Z_q}$, unique up to $\Z_q$-isomorphism, such that the points have intersections $\langle x,z \rangle = \langle y,z \rangle = 0$ and $\langle x,y \rangle=m$.

    \noindent \begin{minipage}{.75\textwidth}
    \begin{proof}
    
    The collection of isomorphism classes over $\Z_q$ of models $M \simeq \mathbb{P}^1_{\Z_q}$ is parametrized by the vertices of the $(q+1)$-regular Bruhat--Tits tree $\mathcal{T}$ whose ends are in bijection with $\PP^1(\Q_q)$. Explicitly, given any such model $M$, every blow-up at one of its $\F_{q}$-rational points, followed by a contraction of the original special fibre, yields a neighbouring model $M'$, see \cite{DT07}. 

    In the model determined by the vertex $v$, the intersection multiplicity of any $x,y \in \PP^1(\Q_q)$ is equal to the distance of $v$ to the path $\{x\to y\}$ in $\mathcal{T}$. So if a vertex $v$ determines a model that satisfies the conditions of the lemma, it must lie on the paths $\{x\to z\}$ and $\{y\to z\}$, and have distance $m$ from the path $\{x\to y\}$. There is a unique vertex $v$ with this property.
    \end{proof}
    \end{minipage} 
    \hfill 
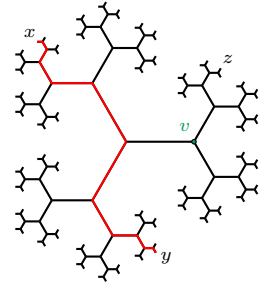
    \begin{minipage}{.25\textwidth}
    \tikzmath{\a = 9; \r = 0.6;}
    \begin{center}
        \begin{tikzpicture}
          [grow cyclic,
           level 1/.style={thick,level distance=\a mm,sibling angle=120},
           level 2/.style={level distance=\a*\r mm,sibling angle=120},
           level 3/.style={level distance=\a*\r*\r mm,sibling angle=120},
           level 4/.style={level distance=\a*\r*\r*\r mm,sibling angle=120},
           level 5/.style={level distance=\a*\r*\r*\r*\r mm,sibling angle=120},
           level 6/.style={level distance=\a*\r*\r*\r*\r*\r mm,sibling angle=120}]
          \coordinate [rotate=0] 
            child foreach \x in {1,2,3}
              {child foreach \x in {1,2}
               {child foreach \x in {1,2}
                {child foreach \x in {1,2}
                {child foreach \x in {1,2}
                {child foreach \x in {1,2}}}}}};
          \node (x) at (-1.28,1.45) {\scriptsize $x$};
          \node (y) at (0.53,-1.55) {\scriptsize $y$};
          \node (z) at (1.35,1.1) {\scriptsize $z$};
    	  \node[draw,fill=ForestGreen,circle,minimum size=.6mm,inner sep=0pt] (A) at (.9,0) {};
          \node[ForestGreen] (v) at (0.78,0.2) {\scriptsize $v$};
          \draw[red,thick] (0,0) -- ++ (120:\a mm) -- ++ (180:\a*\r mm) -- ++ (120:\a*\r*\r mm) -- ++ (60:\a*\r*\r*\r mm) -- ++ (120:\a*\r*\r*\r*\r mm) -- ++ (180:\a*\r*\r*\r*\r*\r mm);
          \draw[red,thick] (0,0) -- ++ (240:\a mm) -- ++ (300:\a*\r mm) -- ++ (0:\a*\r*\r mm) -- ++ (300:\a*\r*\r*\r mm) -- ++ (0:\a*\r*\r*\r*\r mm) -- ++ (300:\a*\r*\r*\r*\r*\r mm);
    
        \end{tikzpicture}
    \end{center}
    \end{minipage}

\end{lemma}

Given a prime $q$, we may use \cref{lem:unique_P1} to ensure that the model $M$ over $\Z_q$ used in our computations \cite{github} yields the correct intersection numbers of all CM points on $\cX_{\rm ns}^+(p)$, by judiciously choosing a triple $x,y,z$ and verifying that they have the correct pairwise intersection numbers on $M$. Let $x,y,z\in\cX_{\ns}^+(p)(\Z)$ be Heegner points with coprime fundamental discriminants that do not all intersect at $q$. Such a triple exists for all pairs of primes $(p,q)\neq (2,2)$, as can be seen from the blue entries of \cref{fig:Xnsp_intersections} given by \cref{thm:B}. For $(p,q)=(2,2)$, we instead take two Heegner points $x,y\in\cX_{\ns}^+(2)(\Z)$ with coprime fundamental discriminants, and $z\in X_{\ns}^+(2)(\Q)$ a point with $j(z) \not \in \Z_2$. Then the corresponding point in $X(\Z_2)$ does not intersect $x$ or $y$. 

\begin{remark}
    The unusually cultured (or unusually-cultured) reader may notice that the intersection numbers on $\cX_{\ns}^+(2)$ are very similar to the factorizations of differences of $j$-invariants: for example, compare the first equation in \cite{GZ85} to entry $(-4,-163)$ in the $p=2$ section of \cref{fig:Xnsp_intersections}. Indeed, the intersection number of $P_1$ and $P_2$ on $\cX_{\ns}^+(2)$ equals the intersection number of $j(P_1)/2^{12}$ and $j(P_2)/2^{12}$ on $\mathbb{P}^1_{\Z}$.
\end{remark}

\begin{remark}\label{rmk:nonheegner}
    There are two distinct points in $X_{\ns}^+(5)(\Q)$ with CM discriminant $-3$, as an elliptic curve $E/\Q$ with $j$-invariant $0$ can be given two non-equivalent non-split Cartan structures  of level $5$ over $\Q$. 
    
    Here is a brief explanation of the source of this extra CM point. There are $10$ non-split Cartan structures on $E[5]$ over $\overline{\Q}$, and $\Aut(E)$ acts on this set. The Heegner structure, induced by the action of the CM order $\End(E)=\Z[\frac12(1+\sqrt{-3})]$ on $E[5]$, is naturally fixed by this action. The other nine structures come in orbits of size $3$, with each orbit defining a single point of $X_{\ns}^+(5)(\overline{\Q})$. For an appropriate basis for $E[5]$, the Heegner structure $A\subset \End(E[5])$ can be identified with the subring of $M_2(\F_5)$ consisting of the elements $\begin{psmallmatrix}
        a & 2b \\ b & a
    \end{psmallmatrix}$ for $a,b\in\F_5$. Now consider the subring $A_2\subset M_2(\F_5)$ consisting of elements $\begin{psmallmatrix}
        a & 3b \\ b & a
    \end{psmallmatrix}$ for $a,b\in\F_5$. One can check that all elements of the $2$-Sylow subgroup of the normalizer $N(A^\times)$ of $A^\times$ also normalize $A_2$, and the elements of order $3$ in $A^\times$ act via $\Aut(E)$, so all $48$ elements of $N(A^\times)$ preserve the $\Aut(E)$-orbit of $A_2$. As the action of $\Gal(\overline{\Q}/\Q)$ on $E[5]$ factors through $N(A^\times)$, we can conclude that the $\Aut(E)$-orbit of $A_2$, and hence the corresponding point on $X_{\ns}^+(5)$, is Galois-invariant. The remaining two orbits determine a Galois-conjugate pair of points in $X_{\ns}^+(5)(\Q(\sqrt{-15}))$.
    
    In \cref{fig:Xnsp_intersections}, we refer to the CM points associated to the Heegner structure $A$ and the additional $\Q$-rational Cartan structure $A_2$ by the labels ``$-3$'' and ``$-3_{2}$'' respectively. Intersections involving $-3_{2}$ may not satisfy \cref{thm:B}, even though the CM discriminant $-3$ satisfies the non-split Heegner condition for $p=5$.
\end{remark}

\begin{remark}
    Aside from the non-Heegner CM point $-3_2$ on $X_{\ns}^+(5)$ from \cref{rmk:nonheegner}, there are also rational CM points with non-Heegner discriminant: 
    \[
    \begin{array}{llll}
    \Delta &=& -4,-7,-8,-12,-16,-28 & \mbox{on $X_{\ns}^+(2)$,}\\
    \Delta &=& -3,-8,-11 & \mbox{on $X_{\ns}^+(3)$,}\\
    \Delta &=& -3 &\mbox{on $X_{\ns}^+(7)$.}
    \end{array}
    \]
    For $p\geq 11$, on the other hand, by the discussion in \cref{rmk:non-fundamental}, every rational CM point is a Heegner point.
\end{remark}

\subsection{Positive genus}

From now on we assume $p\geq 11$. As $\Delta_1, \Delta_2$ range over \cref{eqn:classno1}, we find exactly seven triples $(\Delta_1,\Delta_2,p)$ such that $\mathcal{S}(\Delta_1\Delta_2,1)$ contains a multiple of $p^2$:
\begin{align*}
    3\cdot 11^2&\in \mathcal{S}(16\cdot 163,1), & 18\cdot 11^2&\in \mathcal{S}(67\cdot 163,1),     & 8\cdot 11^2&\in \mathcal{S}(27\cdot 163,1), & 2\cdot 11^2&\in \mathcal{S}(27\cdot 67,1), \\
    2\cdot 13^2&\in \mathcal{S}(11\cdot 163,1), & 12\cdot 13^2&\in \mathcal{S}(67\cdot 163,1), &
    2\cdot 23^2&\in \mathcal{S}(27\cdot 163,1).
\end{align*}
In each case, the multiple of $p^2$ in $\mathcal{S}(\Delta_1\Delta_2,1)$ is coprime to the conductors of $\Delta_1$ and $\Delta_2$, so the intersection numbers of the associated Heegner points on $\cX_{\ns}^+(p)$ can be computed using \cref{prop:non-fundamental}. Alternatively, note that
intersections are only possible at $q=2,3$ by \cref{lem:no_intersect}. The LMFDB \cite{LMFDB} gives explicit models over $\Q$ of $X_{\ns}^+(p)$ for $p\in\{11,13,23\}$. For each pair $(p,q)$, we use Magma to list all $\F_q$-points of the model for $X_{\ns}^+(p)$ and verify smoothness at each point. For $p\in\{11,13\}$ we can then compute intersection multiplicities directly; for $p=23$, we use \cite[Table 8.3]{MS20} to check that the Heegner points for discriminants $-27$ and $-163$ intersect at $2$ with multiplicity $1$.

Finally we consider the case that $\mathcal{S}(\Delta_1\Delta_2,1)$ has no multiples of $p^2$. It follows from \cref{lem:no_intersect} that an intersection of distinct Heegner points in $X_{\rm ns}^+(p)(\Q)$ is only possible if $\Delta_1/\Delta_2$ is an even power of some prime $q$, in which case these points may only intersect at $q$. By \cref{lem:int equals isoms}, the intersection multiplicity can be computed by counting the number of isomorphisms over $W/\fq^n$ that preserve the level structure. By \cref{lem:degenerate}, this is equivalent to counting the number of isomorphisms over $W/\fq^n$ that send $\cO_1$ into the fraction field of $\cO_2$. Observe that this count does not depend on $p$: that is, if $p_0\geq 11$ is another prime inert in the CM orders $\cO_1$ and $\cO_2$, and if $\mathcal{S}(\Delta_1\Delta_2,p_0^2)$ is also empty, then the intersection multiplicities at $q$ of the Heegner points on $X_{\ns}^+(p)$ and $X_{\ns}^+(p_0)$ associated to $\Delta_1$ and $\Delta_2$  are equal. So for the pairs
\[
(\Delta_1, \Delta_2) = (-3,-12), \ (-3,-27), \ (-4,-16),\  (-7,-28),
\]
we only need to compute the intersection multiplicity at $q=2$, $3$, $2$, or $2$ respectively on $X_{\ns}^+(p_0)$ for a single prime $p_0\geq 11$ satisfying the conditions of \cref{lem:degenerate}: for $(-7,-28)$ we may take $p_0=13$, and for the other pairs we may take $p_0=11$. The intersection multiplicity equals that of the corresponding Heegner points on $X_{\ns}^+(p)$, for all $p$ satisfying the same conditions. This yields the final four entries of \cref{fig:Xnsp_intersections}.

\setlength\biblabelsep{0pt}
\printbibliography

\renewcommand{\arraystretch}{1.1}

\begin{table}
\centering
\rotatebox{90}{
\begin{minipage}{0.9\textheight}
\centering
\begin{tabular}{c}
$p=2$:\\
\resizebox{\textwidth}{!}{
$
\begin{array}{|c|cccccccccccc|}
 \hline
 \Delta_1, \Delta_2 &  -4  & -7  & -8  & -11  & -12  & -16  & -19  & -27  & -28  & -43  & -67  & -163  \\
\hline
-3 &  3^{3}  &  3^{3}\,5^{3}  &  5^{3}  & \color{blue}  2^{3}  &  3^{3}\,5^{3}  &  3^{3}\,11^{3}  & \color{blue}  2^{3}\,3^{3}  &  2^{3}\,3^{1}\,5^{3}  &  3^{3}\,5^{3}\,17^{3}  & \color{blue}  2^{6}\,3^{3}\,5^{3}  & \color{blue}  2^{3}\,3^{3}\,5^{3}\,11^{3}  & \color{blue}  2^{6}\,3^{3}\,5^{3}\,23^{3}\,29^{3}  \\
-4 &  &  2^{6}\,3^{6}\,7^{1}  &  2^{7}\,7^{2}  &  7^{2}\,11^{1}  &  2^{6}\,3^{3}\,11^{2}  &  2^{6}\,3^{6}\,7^{2}  &  3^{6}\,19^{1}  &  3^{1}\,11^{2}\,23^{2}  &  2^{6}\,3^{8}\,7^{1}\,19^{2}  &  3^{8}\,7^{2}\,43^{1}  &  3^{6}\,7^{2}\,31^{2}\,67^{1}  &  3^{6}\,7^{2}\,11^{2}\,19^{2}\,127^{2}\,163^{1}  \\
-7 &  &  &  2^{6}\,5^{3}\,7^{1}\,13^{1}  &  7^{1}\,13^{1}\,17^{1}\,19^{1}  &  2^{8}\,3^{3}\,5^{3}\,17^{1}  &  2^{9}\,3^{7}\,7^{1}\,19^{1}  &  3^{7}\,13^{1}\,31^{1}  &  3^{1}\,5^{3}\,17^{1}\,41^{1}\,47^{1}  &  2^{13}\,3^{6}\,5^{3}\,7^{1}\,13^{1}  &  3^{6}\,5^{3}\,7^{1}\,19^{1}\,73^{1}  &  3^{7}\,5^{3}\,7^{1}\,13^{1}\,61^{1}\,97^{1}  &  3^{8}\,5^{3}\,7^{1}\,13^{1}\,17^{1}\,31^{1}\,103^{1}\,229^{1}\,283^{1}  \\
-8 &  &  &  &  7^{2}\,13^{1}  &  2^{6}\,5^{3}\,23^{1}  &  2^{6}\,7^{2}\,23^{1}\,31^{1}  &  13^{1}\,29^{1}\,37^{1}  &  5^{3}\,29^{1}\,53^{1}  &  2^{6}\,5^{3}\,7^{1}\,13^{1}\,31^{1}\,47^{1}  &  5^{3}\,7^{2}\,37^{1}\,61^{1}  &  5^{4}\,7^{2}\,13^{1}\,53^{1}\,109^{1}  &  5^{3}\,7^{2}\,13^{1}\,37^{1}\,101^{1}\,157^{1}\,277^{1}\,317^{1}  \\
-11 &  &  &  &  &  11^{1}\,17^{1}\,29^{1}  &  7^{2}\,19^{1}\,43^{1}  & \color{blue}  2^{4}\,13^{1}  &  2^{4}\,11^{1}\,17^{1}  &  7^{1}\,13^{1}\,41^{1}\,61^{1}\,73^{1}  & \color{blue}  2^{3}\,7^{2}\,19^{1}\,29^{1}  & \color{blue}  2^{5}\,7^{2}\,13^{1}\,41^{1}\,43^{1}  & \color{blue}  2^{3}\,7^{2}\,11^{1}\,13^{1}\,17^{1}\,73^{1}\,79^{1}\,107^{1}\,109^{1}  \\
-12 &  &  &  &  &  &  2^{8}\,3^{3}\,23^{1}\,47^{1}  &  3^{3}\,41^{1}\,53^{1}  &  3^{1}\,5^{3}\,11^{2}\,17^{1}  &  2^{8}\,3^{3}\,5^{3}\,59^{1}\,83^{1}  &  3^{3}\,5^{4}\,29^{1}\,113^{1}  &  3^{3}\,5^{3}\,101^{1}\,137^{1}\,197^{1}  &  3^{3}\,5^{3}\,11^{2}\,17^{1}\,89^{1}\,233^{1}\,293^{1}\,389^{1}  \\
-16 &  &  &  &  &  &  &  3^{7}\,67^{1}  &  3^{1}\,59^{1}\,83^{1}\,107^{1}  &  2^{9}\,3^{6}\,7^{1}\,31^{1}\,103^{1}  &  3^{6}\,7^{2}\,19^{1}\,163^{1}  &  3^{8}\,7^{2}\,11^{3}\,43^{1}  &  3^{7}\,7^{2}\,59^{1}\,67^{1}\,211^{1}\,571^{1}\,643^{1}  \\
-19 &  &  &  &  &  &  &  &  2^{5}\,3^{1}\,29^{1}  &  3^{6}\,13^{1}\,19^{1}\,97^{1}  & \color{blue}  2^{3}\,3^{6}\,37^{1}  & \color{blue}  2^{4}\,3^{7}\,13^{1}\,79^{1}  & \color{blue}  2^{3}\,3^{7}\,13^{1}\,19^{1}\,31^{1}\,37^{1}\,67^{1}\,193^{1}  \\
-27 &  &  &  &  &  &  &  &  &  3^{1}\,5^{4}\,89^{1}\,173^{1}  &  2^{3}\,3^{1}\,5^{3}\,71^{1}  &  2^{4}\,3^{1}\,5^{3}\,53^{1}\,113^{1}  &  2^{3}\,3^{1}\,5^{4}\,11^{2}\,17^{1}\,59^{1}\,137^{1}\,257^{1}  \\
-28 &  &  &  &  &  &  &  &  &  &  3^{8}\,5^{3}\,7^{1}\,157^{1}  &  3^{6}\,5^{3}\,7^{1}\,13^{1}\,41^{1}\,433^{1}  &  3^{6}\,5^{4}\,7^{1}\,13^{1}\,19^{2}\,73^{1}\,241^{1}\,997^{1}  \\
-43 &  &  &  &  &  &  &  &  &  &  & \color{blue}  2^{3}\,3^{6}\,5^{3}\,7^{2}  & \color{blue}  2^{7}\,3^{6}\,5^{3}\,7^{3}\,37^{1}\,433^{1}  \\
-67 &  &  &  &  &  &  &  &  &  &  &  & \color{blue}  2^{3}\,3^{7}\,5^{3}\,7^{2}\,13^{1}\,139^{1}\,331^{1}  \\
\hline
\end{array}
$
}
\end{tabular}

\vspace{6pt}

\begin{tabular}{c}
$p=3$:\\

\begin{scriptsize}
$
\begin{array}{|c|cccccccccc|}
 \hline
 \Delta_1, \Delta_2 &  -4  & -7  & -8  & -11  & -16  & -19  & -28  & -43  & -67  & -163  \\
\hline
-3 &  2^{2}  &  5^{1}  &  2^{2}\,3^{2}\,5^{1}  &  2^{5}\,3^{2}  &  2^{1}\,11^{1}  &  2^{5}  &  5^{1}\,17^{1}  &  2^{6}\,5^{1}  &  2^{5}\,5^{1}\,11^{1}  &  2^{6}\,5^{1}\,23^{1}\,29^{1}  \\
-4 &  & \color{blue}  1  &  2^{3}  &  2^{2}\,11^{1}  &  2^{1}  & \color{blue}  2^{2}  &  3^{2}  & \color{blue}  2^{2}\,3^{2}  & \color{blue}  2^{2}\,7^{2}  & \color{blue}  2^{2}\,7^{2}\,11^{2}  \\
-7 &  &  &  5^{1}\,7^{1}  &  17^{1}  & \color{Green}  3^{1}  & \color{blue}  3^{1}  &  2^{1}\,5^{1}  & \color{blue}  5^{1}\,7^{1}  & \color{blue}  3^{1}\,5^{1}\,13^{1}  & \color{blue}  3^{2}\,5^{1}\,17^{1}\,31^{1}  \\
-8 &  &  &  &  2^{2}\,3^{3}\,13^{1}  &  2^{1}\,23^{1}  &  2^{2}\,29^{1}  &  5^{1}\,47^{1}  &  2^{2}\,5^{1}\,7^{2}  &  2^{2}\,5^{2}\,53^{1}  &  2^{2}\,5^{1}\,101^{1}\,317^{1}  \\
-11 &  &  &  &  &  2^{1}\,7^{2}  &  2^{6}  &  7^{1}\,41^{1}  &  2^{5}\,29^{1}  &  2^{7}\,41^{1}  &  2^{5}\,11^{1}\,17^{1}\,107^{1}  \\
-16 &  &  &  &  &  &  2^{1}\,3^{1}  &  7^{1}  &  2^{1}\,19^{1}  &  2^{1}\,3^{2}\,11^{1}  &  2^{1}\,3^{1}\,59^{1}\,67^{1}  \\
-19 &  &  &  &  &  &  &  13^{1}  & \color{blue}  2^{5}  & \color{blue}  2^{6}\,3^{1}  & \color{blue}  2^{5}\,3^{1}\,13^{1}\,19^{1}  \\
-28 &  &  &  &  &  &  &  &  3^{2}\,5^{1}  &  5^{1}\,41^{1}  &  5^{2}\,13^{1}\,73^{1}  \\
-43 &  &  &  &  &  &  &  &  & \color{blue}  2^{5}\,5^{1}  & \color{blue}  2^{7}\,5^{1}\,37^{1}  \\
-67 &  &  &  &  &  &  &  &  &  & \color{blue}  2^{5}\,3^{1}\,5^{1}\,7^{2}  \\
\hline
\end{array}
$
\end{scriptsize}
\end{tabular}

\vspace{6pt}

\begin{tabular}{cc}
$p=5$: 
& 
$p=7$: \\
\begin{scriptsize}
$
\begin{array}{|c|ccccccccc|}
 \hline
 \Delta_1, \Delta_2 &  -3_{2}  & -7  & -8  & -12  & -27  & -28  & -43  & -67  & -163  \\
\hline
-3 &  1  & \color{blue}  1  & \color{blue}  1  &  2^{1}  &  3^{1}  & \color{Green}  1  & \color{blue}  1  & \color{blue}  2^{3}  & \color{blue}  2^{3}  \\
-3_{2} &  &  1  &  2^{1}  &  3^{1}  &  2^{2}  &  3^{1}  &  2^{2}  &  11^{1}  &  29^{1}  \\
-7 &  &  & \color{blue}  1  & \color{Green}  1  & \color{Green}  1  &  2^{1}  & \color{blue}  3^{1}  & \color{blue}  3^{1}  & \color{blue}  3^{1}\,7^{1}  \\
-8 &  &  &  &  1  & \color{Green}  2^{1}  &  1  & \color{blue}  2^{1}  & \color{blue}  5^{1}  & \color{blue}  13^{1}  \\
-12 &  &  &  &  &  1  &  3^{1}  & \color{Green}  5^{1}  &  2^{1}  &  2^{1}\,17^{1}  \\
-27 &  &  &  &  &  & \color{Green}  5^{1}  & \color{Green}  2^{3}  &  1  &  5^{1}\,11^{1}  \\
-28 &  &  &  &  &  &  &  1  &  13^{1}  &  5^{1}  \\
-43 &  &  &  &  &  &  &  & \color{blue}  3^{1}\,7^{1}  & \color{blue}  3^{1}  \\
-67 &  &  &  &  &  &  &  &  & \color{blue}  2^{4}\,3^{2}  \\
\hline
\end{array}
$
\end{scriptsize}

&

\begin{scriptsize}
$
\begin{array}{|c|ccccccc|}
 \hline
 \Delta_1, \Delta_2 &  -4  & -8  & -11  & -16  & -43  & -67  & -163  \\
\hline
-3 &  3^{1}  &  2^{1}  &  1  &  1  &  5^{1}  &  2^{2}  &  3^{1}  \\
-4 &  &  1  & \color{blue}  1  &  2^{1}  & \color{blue}  1  & \color{blue}  1  & \color{blue}  3^{2}  \\
-8 &  &  & \color{blue}  1  &  1  & \color{blue}  1  & \color{blue}  2^{1}  & \color{blue}  5^{1}  \\
-11 &  &  &  & \color{Green}  1  & \color{blue}  2^{1}  & \color{blue}  1  & \color{blue}  2^{2}  \\
-16 &  &  &  &  & \color{Green}  3^{1}  & \color{Green}  3^{1}  &  1  \\
-43 &  &  &  &  &  & \color{blue}  3^{1}  & \color{blue}  2^{1}\,7^{1}  \\
-67 &  &  &  &  &  &  & \color{blue}  13^{1}  \\
\hline
\multicolumn{8}{c}{}\\
\multicolumn{8}{c}{}\\
 \end{array}
$
\end{scriptsize}
\end{tabular}

\vspace{6pt}

\begin{tabular}{c}
$p\geq 11$ (all intersection numbers $\neq 1$): \\
\begin{scriptsize}
    $
	\begin{array}{|c|cccc|cc|c|cc|c|c|}
		\hline
		p & \multicolumn{4}{c|}{11} & \multicolumn{2}{c|}{13} & 23 & \multicolumn{2}{c|}{2\bmod 3} & 3\bmod 4 & 3,5,6\bmod 7 \\
        \Delta_1 & -16 & -27 & -27 & -67 & -11 & -67 & -27 & -3 & -3 & -4 & -7 \\
        \Delta_2 & -163 & -67 & -163 & -163 & -163 & -163 & -163 & -12 & -27 & -16 & -28\\
        \hline
        & \color{Green} 3^1 & \color{Green} 2^1 & \color{Green} 2^2 & \color{blue} 2^1 & \color{blue} 2^1 & \color{blue} 3^1 & \color{Green} 2^1 & \color{BrickRed} 2^1 & \color{BrickRed} 3^1 & \color{BrickRed} 2^1 & \color{BrickRed} 2^1 \\
        \hline
	\end{array}
	$
\end{scriptsize}
\end{tabular}
\end{minipage}
}

    \caption{Intersection numbers $\exp(\langle P_1,P_2\rangle)$ of CM points $P_1,P_2$ in $X_{\ns}^+(p)(\Q)$, coloured {\color{blue}blue} when \cref{thm:B} applies, {\color{Green}green} when \cref{prop:non-fundamental} applies (but not \cref{thm:B}), and {\color{BrickRed}red} when \cref{lem:degenerate} applies. In all other cases it is coloured black.}
    \label{fig:Xnsp_intersections}
\end{table}

\end{document}

%% file: preamble.tex
\usepackage[dvipsnames]{xcolor}
\usepackage{amsthm,amssymb,amsmath,amsfonts,mathrsfs,amscd,yfonts,stmaryrd}
\usepackage{url,tikz,enumitem,caption,stmaryrd,comment,mathrsfs,marginnote,tikz-cd,mathtools,pgfplots}
\usepackage{dcolumn,booktabs,lscape}
\usepackage[utf8]{inputenc}
\usepackage[T1]{fontenc}
\usepackage{libertine}
\usetikzlibrary{matrix,positioning,arrows,shapes,chains,calc,automata,backgrounds}
\usetikzlibrary{decorations.pathmorphing}
\usetikzlibrary{arrows,automata,positioning}
\usetikzlibrary{decorations.markings,positioning}
\usetikzlibrary{arrows.meta}
\usetikzlibrary{arrows,automata,positioning,arrows.meta,math,calc,trees,optics}
\usepackage{tgcursor}

\usetikzlibrary{decorations.pathmorphing}
\usetikzlibrary{arrows,automata,positioning}
\usetikzlibrary{calc}
\usepackage{hyperref}
\usepackage{framed}
\usepackage{rotating}
\usepackage[nameinlink]{cleveref}

\usepackage[backref,style=alphabetic,doi=false,url=false,isbn=false,maxbibnames=99]{biblatex}
\AtBeginBibliography{\footnotesize}
\DefineBibliographyStrings{english}{%
	backrefpage = {$\uparrow$},
	backrefpages = {$\uparrow$},
}

\DeclareFieldFormat{postnote}{#1}
\DeclareFieldFormat{multipostnote}{#1}

\let\oldtocsection=\tocsection
\let\oldtocsubsection=\tocsubsection
\renewcommand{\tocsection}[2]{\vspace{-.3em}\hspace{0em}\oldtocsection{#1}{#2}}
\renewcommand{\tocsubsection}[2]{\vspace{-.3em}\hspace{2em}\oldtocsubsection{#1}{#2}}

\newlength{\proofmargin}
\DeclareMathOperator{\SL}{\mathrm{SL}}
\DeclareMathOperator{\GL}{\mathrm{GL}}

\DeclareMathOperator{\Gal}{\mathrm{Gal}}

\DeclareMathOperator{\OO}{\mathrm{O}}

\DeclareMathOperator{\Q}{\mathbb{Q}}
\DeclareMathOperator{\F}{\mathbb{F}}
\DeclareMathOperator{\Z}{\mathbb{Z}}

\DeclareMathOperator{\A}{\mathbb{A}}
\DeclareMathOperator{\C}{\mathbb{C}}
\DeclareMathOperator{\PP}{\mathbb{P}}

\DeclareMathOperator{\cC}{\mathcal{C}}

\DeclareMathOperator{\cO}{\mathcal{O}}
\DeclareMathOperator{\cM}{\mathcal{M}}

\DeclareMathOperator{\cS}{\mathcal{S}}
\DeclareMathOperator{\cX}{\mathscr{X}}
\DeclareMathOperator{\cL}{\mathcal{L}}
\DeclareMathOperator{\cP}{\mathcal{P}}
\DeclareMathOperator{\cE}{\mathcal{E}}

\DeclareMathOperator{\fa}{\mathfrak{a}}

\DeclareMathOperator{\fq}{\mathfrak{q}}

\DeclareMathOperator{\spl}{\scriptstyle \mathrm{sp}}
\DeclareMathOperator{\ns}{\scriptstyle \mathrm{ns}}
\DeclareMathOperator{\Car}{\scriptstyle \mathrm{Car}}
\DeclareMathOperator{\lra}{\longrightarrow}
\DeclareMathOperator{\bE}{\mathbf{E}}
\DeclareMathOperator{\End}{\mathrm{End}}
\DeclareMathOperator{\Aut}{\mathrm{Aut}}

\newtheorem{theorem}{Theorem}[section]
\newtheorem{prop}[theorem]{Proposition}

\newtheorem{lemma}[theorem]{Lemma}
\newtheorem{corollary}[theorem]{Corollary}
\newtheorem{definition}[theorem]{Definition}
\newtheorem{thmx}{Theorem}
\theoremstyle{remark}
\newtheorem{remark}[theorem]{Remark}
\newtheorem{example}[theorem]{Example}

\crefname{section}{\S}{\S\S}
\crefname{subsection}{\S}{\S\S}
\Crefname{section}{Section}{Sections}
\Crefname{subsection}{Section}{Sections}
\crefformat{equation}{#2(#1)#3}
\crefname{prop}{proposition}{propositions}

\AddToHook{env/prop/begin}{\crefalias{theorem}{prop}}
\AddToHook{env/conjecture/begin}{\crefalias{theorem}{conjecture}}
\AddToHook{env/lemma/begin}{\crefalias{theorem}{lemma}}
\AddToHook{env/corollary/begin}{\crefalias{theorem}{corollary}}
\AddToHook{env/definition/begin}{\crefalias{theorem}{definition}}
\AddToHook{env/remark/begin}{\crefalias{theorem}{remark}}
\AddToHook{env/example/begin}{\crefalias{theorem}{example}}

\makeatletter
\renewenvironment{proof}[1][\proofname]%
{%
\par\pushQED{\qed}\normalfont\topsep6\p@\@plus6\p@\relax%
\begin{list}{}{\rightmargin=8pt\leftmargin=\proofmargin}%
  \item[\hskip\labelsep\bfseries#1\@addpunct{.}]\ignorespaces
}{%
\popQED\end{list}\@endpefalse%
}%
\makeatother

\usepackage{multiaudience}
\SetNewAudience{authors}
\SetNewAudience{shareable}